%% file: pf_general.tex
\documentclass[11pt,reqno]{amsart}
\usepackage[margin=1.25in]{geometry}

\usepackage{mathrsfs}
\usepackage{amsfonts}
\usepackage{latexsym,epsfig}
\usepackage{mathabx}
\usepackage{amsmath, amscd, amsthm,amssymb}
\usepackage[pdftex]{color}
\usepackage{comment}
\usepackage{marginnote}
\usepackage[colorlinks,citecolor=blue,linkcolor=red]{hyperref}
\usepackage[nameinlink]{cleveref}

\newcommand{\RR}[0]{\mathbb{R}}
\newcommand{\HH}[0]{\mathbb{H}}

\newcommand{\Z}{\mathbb{Z}}

\newcommand{\A}[0]{\mathcal{A}}
\newcommand{\PP}[0]{\mathcal{P}}

\newcommand{\R}{\mathbb{R}}

\newcommand{\M}{\mathring{M}}
\newcommand{\sing}{\mathrm{sing}}

\newcommand{\diam}{\mathrm{diam}}

\newcommand\vmax{\vee}
\newcommand\vmin{\wedge}
\newcommand{\CH}{{\operatorname{CH}}}
\newcommand{\side}[1]{\mathcal{D}_{#1}}

\usepackage{hyperref}

\numberwithin{equation}{section}
\newtheorem{theorem}{Theorem}[section]
\newtheorem{lemma}[theorem]{Lemma}
\newtheorem{proposition}[theorem]{Proposition}
\newtheorem{corollary}[theorem]{Corollary}

\theoremstyle{definition}
\newtheorem{remark}[theorem]{Remark}

\newtheorem{claim2}{Claim}
\newcommand{\FF}{\bf F}

\newcommand{\cF}{\mathring{F}}

\newcommand{\N}{\mathcal{N}}

\newcommand{\define}[1]{\textbf{#1}}

\newcommand{\union}{\cup}
\newcommand{\intersect}{\cap}
\newcommand{\boundary}{\partial}

\newcommand\tsim{\kern-.4em\sim}
\newcommand\til{\widetilde}
\newcommand\ep{\epsilon}

\newcommand\thull{{\mathbf t}}

\newcommand\ssm{\smallsetminus}

\newcommand{\cX}{\mathring{X}}

\newcommand{\cM}{\mathring{M}}
\newcommand{\cS}{\mathring{S}}

\renewcommand{\int}{\mathrm{int}}
\newcommand{\cl}{\mathrm{cl}}
\newcommand{\mc}{\mathcal}
\newcommand{\wt}{\widetilde}

\long\def\realfig#1#2{
  \begin{figure}[htbp]
    \begin{center}
\graphicspath{ {./onelipschitz/} }
\includegraphics {#1}
\caption[#1]{#2}
\label{#1}
    \end{center}
\end{figure}}

\begin{document}
\title[Projections and fibrations]{Subsurface distances for hyperbolic $3$--manifolds fibering over the circle}

\author[Y. N. Minsky]{Yair N. Minsky}
\address{Department of Mathematics\\ 
Yale University}
\email{\href{mailto:yair.minsky@yale.edu}{yair.minsky@yale.edu}}
\author[S.J. Taylor]{Samuel J. Taylor}
\address{Department of Mathematics\\ 
Temple University}
\email{\href{mailto:samuel.taylor@temple.edu}{samuel.taylor@temple.edu}}
\date{\today}

\begin{abstract}
For a hyperbolic fibered $3$--manifold $M$, we prove results that uniformly relate the structure of surface projections as one varies the fibrations of $M$.
This extends our previous work from the fully-punctured to the general case.
\end{abstract}

\maketitle

\setcounter{tocdepth}{1}
\tableofcontents

\input{introduction}

\input{background}

\input{veering}

\input{immersions}

\input{bounds}

\input{pocketbound}

\input{dichotomy}


\bibliography{pfnew.bib}
\bibliographystyle{amsalpha}

\end{document}

%% file: introduction.tex

\section{Introduction}
Let $M$ be a hyperbolic $3$-manifold and let
$S$ be a fiber in a fibration of $M$ over the circle.
The corresponding monodromy is a pseudo-Anosov homeomorphism $\phi \colon S \to S$ and comes equipped with invariant  stable and unstable laminations $\lambda^\pm$ on $S$. Let $N_S \to M$ be the infinite cyclic covering of $M$ corresponding to $S$.

As a consequence of the proof of Thurston's Ending Lamination Conjecture, Minsky \cite{ECL1} and Brock--Canary--Minsky \cite{ELC2} develop combinatorial tools to study the geometry of a hyperbolic manifold homeomorphic to $S \times \RR$. 
Applying their work to the special case of $N_S \cong S \times \RR$
explains how the geometry of $N_S$, and hence that of $M$, is \emph{coarsely} determined by combinatorial data associated to the pair of laminations $\lambda^\pm$. In particular, using only the pair $\lambda^\pm$, a combinatorial model of $N_S$ (called the \emph{model manifold}) is constructed and this model is shown to be biLipschitz to $N_S$, where the biLipschitz constant depends only on the complexity of the surface $S$. For this, one of the main combinatorial tools are the Masur--Minsky subsurface projections \cite{MM2}, which associate to each subsurface $Y \subset S$ a \emph{subsurface projection distance} $d_Y(\lambda^-,\lambda^+)$ measuring the complexity of $\lambda^\pm$ as seen from $Y$. In fact, subsurface projections have proven to be useful in several settings \cite{Rafi,BKMM, masur2013geometry} and have been generalized in many directions \cite{BBFquasi, bestvina2014subfactor, behrstock2017hierarchically, sisto2019largest}.

These developments suggest the following outline for studying the geometry of a hyperbolic fibered $3$-manifold $M$: apply the model manifold machinery to an infinite cyclic cover of $M$ associated to a fiber and use this to make conclusions about the structure of $M$. In fact, because of Agol's resolution of the Virtual Fibering Conjecture \cite{AgolVHC}, this simple idea generalizes to any hyperbolic manifold by first passing to a finite sheeted cover which fibers over the circle. 

Unfortunately, this approach is too na\"ive for a number of reasons, perhaps the most important of which is that the complexity of a fiber in the appropriate cover is \emph{not} known at the outset.
Indeed, even when a fibered manifold $M$ is fixed, if $\mathrm{dim}(H^1(M; \RR)) \ge 2$ then $M$ fibers in infinitely many ways, and the complexities of the corresponding fibers are necessarily unbounded. Since the bilipschitz constants in the Model Manifold Theorem depend on the complexity of the underlying surface, 
this approach goes nowhere without a precise understanding of how
the constants relating geometry to combinatorics vary as the surface changes.

To salvage this approach, one would like control over how the tools at the center of the construction depend on complexity. The purpose of this paper is to give such uniform, explicit control on subsurface projection distances as one varies the fibers within a fixed fibered manifold. This extends our previous work \cite{veering1} that handled the special case of \emph{fully-punctured} fibered manifolds (see below for details).


\subsection*{Main results}
Recall that the fibrations of a manifold $M$ are organized into finitely many ``fibered
faces'' of the unit ball in $H^1(M, \RR)$ of the Thurston norm \cite{thurston1986norm}, where
each fibered face $\FF$ has the property that all primitive integral classes in the open cone $\RR_+\FF$ represent a fiber (see \Cref{sec:tnorm}).
Associated to each  fibered face  is a pseudo-Anosov flow which is
transverse to every fiber represented in $\RR_+\FF$ \cite{fried1982geometry}.

Our first main result bounds the size and projection distance for all subsurfaces of all
fibers over a fixed fibered face $\FF$. The constant $D$ in the statement is no more than $15$ (see \Cref{distance and intersection}) and $|\chi'(Y)| = \max\{|\chi(Y)|, 1\}$.

\begin{theorem}[Bounding projections for $M$] \label{th:intro_1}
Let $M$ be a hyperbolic fibered $3$-manifold with fibered face $\FF$. Then for any fiber $S$ contained in $\R_+  \FF$ and any subsurface $Y$ of $S$
\[
 |\chi'(Y)| \cdot \big (d_Y(\lambda^- ,\lambda^+) -16D \big ) \le 2D \: \vert \FF \vert,
\]
where $|\FF|$ is a constant depending only on $\FF$.
\end{theorem}

In particular, subsurface projections are uniformly bounded over the fibered face $ \FF$ as are the complexities of subsurfaces whose projection distances are greater than $16 D$. Note that since $M$ has only finitely many fibered faces, this bounds the size of all subsurface projections among all fibers of $M$.

\medskip

Second, we relate subsurfaces of different fibers in the same fibered face of $M$. Note
that here the constants involved do not depend on the manifold $M$. 

\begin{theorem}[Subsurface dichotomy] \label{th:intro_2}
Let $M$ be a hyperbolic fibered $3$-manifold and let $S$ and $F$ be fibers of $M$ which are contained in the same fibered cone. If $W \subset F$ is a subsurface of $F$ then either $W$ is 
homotopic, through surfaces transverse to the associated flow,
to an \emph{embedded} subsurface $W'$ of $S$ with 
$$d_{W'} (\lambda^-,\lambda^+) = d_W(\lambda^-,\lambda^+) $$
or the fiber $S$ satisfies
$$9D \cdot |\chi(S)| \ge  d_W(\lambda^-,\lambda^+) -16D.$$
\end{theorem}

Along the way to establishing our main theorems, we prove several results that may be independently interesting. While most of these concern the connection between the manifold $M$ and the veering triangulation of the associated fully-punctured manifold (see the next section for details), we also obtain information about subsurface projections to \emph{immersed} subsurfaces.

\medskip

For a finitely generated subgroup $\Gamma < \pi_1(S)$, let $S_\Gamma
\to S$ be the corresponding cover. 
If $Y\subset S_\Gamma$ is a
compact core, the covering map restricted to $Y$ is an immersion and we
say that $Y\to S$ {\em corresponds to $\Gamma<\pi_1(S)$.}  
Lifting to the cover induces a (partially defined) map of curve and arc graphs which we
denote $\pi_Y \colon \A(S)\to \A(Y)$.
(When $\Gamma$ is cyclic, we set $\A(Y)$ to be the annular curve graph $\A(S_\Gamma)$ as usual.)
Note that these constructions depend on
$\Gamma$ but not on the choice of $Y$ and that $\pi_Y$ agrees with the usual
subsurface projection when $Y \subset S$ is an embedded subsurface.

\begin{theorem}[Immersions to covers]\label{th:intro_3}
There is a constant $M\le 38 $ satisfying the following: 
Let $S$ be a surface and let $Y \to S$ be an immersion corresponding
to a finitely generated $\Gamma < \pi_1(S)$. Then either
\begin{itemize}
\item there is a subsurface $W \subset S$ so that $Y \to S$ factors (up to homotopy) through a finite sheeted covering $Y \to W$, or
\item the diameter of the entire projection of $\A(S)$ to $\A(Y)$ is bounded by $M$.
\end{itemize}
\end{theorem}

The novelty of
\Cref{th:intro_3}
 is that the constant $M \le 38$ is explicit and uniform over all surfaces and immersions. Previously, Rafi and Schleimer proved that for any finite cover $\til S \to S$ there is a constant $T\ge0$ (depending on $\til S$ and $S$) such that if $Y \subset \til S$ is a subsurface with $d_Y(\alpha, \beta) \ge T$ for $\alpha,\beta \in \A(S)$, then $Y$ covers a subsurface of $S$ \cite[Lemma 7.2]{rafi2009covers}.

\subsection*{Relation to our previous work}
Given a fibered face $\FF$ of $M$ and its associated pseudo-Anosov flow,
the stable/unstable laminations $\Lambda^\pm$ of the flow intersect each fiber to give the laminations associated to its monodromy. Removing the singular orbits of the flow produces the \emph{fully-punctured} manifold $\cM$ associated to the face $\FF$. If $\phi \colon S \to S$ is the monodromy of some fiber $S$ representing a class in $\RR_+\FF$, then $\cM$ is the mapping torus of the surface $\cS$ obtained from $S$ by puncturing at the singularities of $\phi$. The fibered face of $\cM$ containing $\cS$ is denoted $\mathring{\FF}$ and the inclusion $\cM \subset M$ induces an injective homomorphism $H^1(M \; \RR) \hookrightarrow H^1(\cM \; \RR)$ mapping $\RR_+\FF$ into $\RR_+\mathring{\FF}$.

In our previous work \cite{veering1}, we restricted our study 
of subsurface projections in fibered manifolds 
to the fully-punctured settings. When $\cM$ is fully-punctured, it admits a canonical \emph{veering triangulation} $\tau$ \cite{agol2011ideal, gueritaud} associated to the fibered face $\mathring{\FF}$. 
We found that the combinatorial structure of this triangulation encodes the hierarchy of subsurface projections for each fiber $\cF$ in $\RR_+\mathring{\FF}$. 
As a result, we established versions of \Cref{th:intro_1,th:intro_2} in that restricted setting (though with better constants than available in general). In fact, when the fibered manifold is fully-punctured there are additional surprising connections between the veering triangulation and the curve graph.
For example, a fiber $\cF$ of $\cM$ is necessarily a punctured surface, and edges of the triangulation $\tau$ (when lifted to the cover of $\cM$ corresponding to $\cF$) form a subset of the arc graph $\A(\cF)$. 
This subset is \emph{geodesically connected} in the sense that for any pair of arcs of $\cF$ coming from edges of $\tau$ there is a geodesic in $\A(\cF)$ joining them consisting entirely of veering edges \cite[Theorem 1.4]{veering1}. Such a result cannot have a precise analog if, for example, the manifold $M$ is closed.

\smallskip

In this paper, we extend our study to general (e.g. closed) hyperbolic fibered manifolds. The main difficulty here is that these manifolds do not admit veering triangulations. So our approach is to start with an arbitrary fibered manifold $M$ and consider the veering triangulation of the associated fully-punctured manifold $\cM$. 
(For example, the constant $|\FF|$ appearing in \Cref{th:intro_1} is precisely the number of tetrahedra of the veering triangulation of $\cM$ associated to $\mathring{\FF}$.)
Unfortunately, results about subsurface projections to fibers of $M$ do not directly imply the corresponding statements in $\M$. Instead, we develop tools to relate \emph{sections} of the veering triangulation (i.e. ideal triangulations of the fully-punctured fiber by edges of the veering triangulation) to subsurface projections in the fibers of $M$.  

\subsection*{Summary of paper}
In \Cref{sec:background} we present background material. In particular, we summarize the definition of the veering triangulation (\Cref{veering defs}) and recall the main constructions from \cite{veering1} that connect the structure of the veering triangulation on $\cS \times \RR$ to subsurface representatives in $\cS$ (\Cref{sec:compatibility}).

\Cref{sec:veering} introduces the lattice structure of \emph{sections} of the veering triangulation. It concludes with \Cref{sec:proj_tau} which details how sections (which are ideal triangulations of the fully-punctured surface $\cS$) are used to define projections to the curve graph of subsurfaces of the original surface $S$. This is followed by \Cref{sec:immersions} where \Cref{th:intro_3} is proven. This section does not use veering triangulations and can be read independently from the rest of the paper.

In \Cref{sec:bounds}, we prove two estimates that relate the veering triangulation of the fully-punctured manifold $\cM$ to fibers of $M$. The first (\Cref{prop:closed_distance}) shows that for each subsurface $Y$ of $S$, there are top and bottom sections of $\cS \times \RR$ which project close to the images of $\lambda^\pm$ in $\A(Y)$. The second (\Cref{lem:slowed_progress}) shows that these projections to $\A(Y)$ move slowly from $\pi_Y(\lambda^-)$ to $\pi_Y(\lambda^+)$ depending on the size of $Y$. Both of these estimates are needed for proofs of \Cref{th:intro_1,th:intro_2}.

Finally, in \Cref{sec:uniform_bounds,sec:dichotomy},  \Cref{th:intro_1,th:intro_2} are proven. The bulk of the proof of \Cref{th:intro_1} involves building a simplicial \emph{pocket} for the subsurface $Y$ that embeds into the veering triangulation of $\cM$ whose ``width'' is approximately $|\chi(Y)|$ and whose ``depth'' is at least $d_Y(\lambda^-,\lambda^+)$. For \Cref{th:intro_2}, we show that if a subsurface $W$ of a fiber $F$ is not homotopic into another fiber $S$ (in the same fibered cone) then, after puncturing along singular orbits of the flow, a section (triangulation) of $\cS$ contains many edges, proportional to the ``depth'' of the pocket for $W$. For each of these arguments, the difficulty lies in the fact that we are extracting information about the original manifold $M$ and projections to subsurfaces of its fibers by relying on the veering triangulation of $\cM$, which a priori only records information about projections to its fully-punctured fibers.

\subsection*{Acknowledgments}
Minsky was partially supported by NSF grants DMS-1610827 and DMS-2005328, and Taylor was partially supported by NSF grants DMS-1744551 and DMS-2102018 and the Sloan Foundation.

%% file: background.tex

\section{Background}
\label{sec:background}
Here we record basic background that we will need throughout the paper. We begin with some material that follows easily from standard facts about curve graphs and then recall the definition of the veering triangulation. We conclude by reviewing results from \cite{veering1} which develop connections between the two.

\subsection{Curve graph facts and computations}
The arc and curve graph $\A(Y)$ for a compact surface $Y$ is the graph whose
vertices are homotopy classes of essential simple closed curves and proper arcs. 
Edges join vertices precisely when the vertices have disjoint representatives on $Y$.
Here, essential curves/arcs are those which are not homotopic (rel endpoints) to a point or into the boundary.

If $Y$ is not an annulus,
homotopies of arcs are assumed to be homotopies through maps sending boundary to boundary. This is equivalent to considering proper embeddings $\R \to \int(Y)$ into the interior of $Y$ up to proper homotopy, and we often make use of this perspective.
When $Y$ is an annulus the homotopies are also required to fix the endpoints.
We consider $\A(Y)$ as a metric space by using the graph metric, although we usually only consider distance between vertices. For additional background, the reader is referred to \cite{MM2} and \cite{ECL1}.

\smallskip

If $Y\subset S$ is an essential subsurface (i.e. one that is $\pi_1$--injective and contains an essential curve), we have subsurface projections $\pi_Y(\lambda)$ which are
defined for simplices $\lambda\subset \A(S)$ that intersect $Y$
essentially, otherwise the projection is defined to be empty. Namely, after lifting $\lambda$ to the
cover $S_Y$ associated to $\pi_1(Y)$, 
we obtain a collection of
properly embedded disjoint essential arcs and curves, 
which determine a simplex of $\A(Y): = \A(S_Y)$.
We let $\pi_Y(\lambda)$ be the union of these vertices. The same definition applies to a lamination $\lambda$ that intersects $\partial Y$ essentially.


When $Y$ is an annulus these arcs have natural endpoints
coming from the standard compactification of $\til S = \HH^2$ by a
circle at infinity. We remark that $\pi_Y$ does not depend on any choice
of hyperbolic metric on $S$.

When $Y$ is not an annulus and $\lambda$ and $\boundary Y$ are in
minimal position, we can also identify $\pi_Y(\lambda)$ with the
isotopy classes of components of $\lambda\intersect Y$. 


When $\lambda,\lambda'$ are two
arc/curve systems or laminations, we denote by $d_Y(\lambda,\lambda')$ the 
diameter of the union of their images in $\A(Y)$,
 that is
$$
d_Y(\lambda,\lambda') = \diam_{\A(Y)}(\pi_Y(\lambda)\union\pi_Y(\lambda')).
$$


\subsubsection{An ordering on subsurface translates} \label{sec:order}
Here we prove a lemma that will be needed in \Cref{sec:uniform_bounds}. It establishes an
ordering on translates of a fixed subsurface under a pseudo-Anosov map by appealing to a
more general ordering of Behrstock--Kleiner--Minsky--Mosher \cite{BKMM}, as refined in
Clay--Leininger--Mangahas \cite{CLM}.

Fix a pseudo-Anosov $\phi \colon S \to S$ with stable and unstable laminations $\lambda^+$ and $\lambda^-$, respectively. Our convention is that $\phi$'s unstable lamination $\lambda^-$ is its attracting fixed point on $\mathcal{PML}$.

\begin{lemma}\label{phi n of boundary Y} \label{lem:order}
If $d_Y(\lambda^-,\lambda^+) \ge 20$, then for any $n\ge 1$ with $\pi_Y(\phi^n(\partial Y)) \neq \emptyset$,
\[
d_Y(\phi^n(\boundary Y), \lambda^-) \le 4.
\]
\end{lemma}

\begin{proof}
First consider the set of subsurfaces $\mc S = \{Y: d_Y(\lambda^-,\lambda^+) \ge
20\}$. If $Y,Z$ are members of $\mc S$
that overlap nontrivially then, following \cite{CLM}, we say $Y\prec Z$ if
\[
d_Y(\boundary Z,\lambda^+) \le 4.
\]
According to \cite[Proposition 3.6]{CLM}, this is equivalent to the condition that $d_Z(\boundary Y,\lambda^-) \le 4$, and any two overlapping $Y,Z \in \mc S$ are ordered.
Moreover by \cite[Corollary 3.7]{CLM}, $\prec$ is a strict partial order on $\mc S$.

Returning to our setting, suppose that $Y\in \mc S$ and that
$\phi^n(Y)$ and $Y$ overlap for some $n\ge 1$. Consider the sequence $Y_i =
\phi^{in}(Y)$.

Now, we know that $\boundary Y_i \to \lambda^-$ in $\mathcal{PML}$ as $i\to+\infty$. This implies
that for large enough $i$ we have $d_Y(\boundary Y_i,\lambda^-) \le
1$, and hence $Y_i \prec Y_0$.

On the other hand if $Y_0 \prec Y_1$
then, since $\phi$ preserves $\lambda^\pm$,
we have $Y_{i} \prec Y_{i+1}$ for all $i$. Since $\prec$ is transitive, this would imply that
$Y_0 \prec Y_i$, a contradiction.

Since $Y_0$ and $Y_1$ are ordered, we must have $Y_1 \prec Y_0$.
Hence, $d_Y(\phi^n(\boundary Y) ,\lambda^-) \le 4$, which is what we wanted to prove.
\end{proof}

\subsubsection{Distance and intersection number} \label{distance and intersection}
For an orientable surface $S$ with genus $g\ge 0$ and $p\ge 0$ punctures, set $\zeta =
\zeta(S) = 2g+p-4 = |\chi(S)| -2$.
The following lemma of Bowditch will be important in making our estimates uniform over complexity. Asymptotically stronger, yet less explicit, bounds were first proven by Aougab \cite{Aougab1}. 

\begin{lemma}[Bowditch \cite{uniformBo}] \label{lem:effective_B}
For any integer $n\ge0$ and \emph{curves} $\alpha,\beta \in \A(S)^{(0)}$ with $\zeta =\zeta(S)$,  
\[ 2^n \cdot i(\alpha,\beta) \le \zeta^{n+1} \implies d_S(\alpha,\beta) \le 2(n+1).\]
\end{lemma}

If the surface $S$ is punctured, then for any arcs $a$ and $b$ in $\A(S)^{(0)}$ there is a curve $\alpha \in \A(S)^{(0)}$ disjoint from $a$ and a curve $\beta \in \A(S)^{(0)}$ disjoint from $b$ such that $i(\alpha,\beta) \le 4 \cdot i(a,b) +4$. These curves are constructed using the standard projection from the arc graph to the curve graph: $\alpha$ is a boundary component of a neighborhood of $a \cup P_1 \cup \ldots \cup P_k$, where $P_i$ is a loop circling the $i$th puncture of $S$.

\smallskip

Applying \Cref{lem:effective_B} together with the above observation, we compute that for curves/arcs $a,b \in \A(S)^{(0)}$ and  $\zeta(S) \ge 3$,
\begin{align*}
d_S(a,b)
< 6+  2 \cdot \frac{\log(2 i(a, b) +1)}  {\log(\zeta/2))},
\end{align*}
where $\log = \log_2$. 
We also recall the standard complexity independent inequality (see e.g. \cite{He, schleimer2006notes})
\[
d_S(a,b) \le 2\log(i(a,b)) +2,
\]
where $i(a,b) >0$.
Using these inequalities, straightforward computations (which we omit) show that for any curves/arcs $a,b \in \A(S)^{(0)}$:
\begin{equation}\label{eqn:i bound 8}
  i(a,b) \le 8|\chi(S)| +4  \implies d_S(a,b) \le 15,
\end{equation}
and
\begin{equation}\label{eqn:i bound 32}
i(a,b) \le 32|\chi(S)| +8 \implies d_S(a,b) \le 18.
\end{equation}

We remark that the above mentioned work of Aougab \cite{Aougab1} implies that if $i(a,b) \le K |\chi(S)|$, then $d_S(a,b) \le 3$, so long as $|\chi(S)|$ is sufficiently large (depending on $K$).

\subsubsection{Proper graphs} \label{sec:prop_graph}
Throughout the paper we will use curve graph tools to study objects that arise from (partially) ideal triangulations of surfaces. To do this, we introduce the notion of a proper graph.

A \define{proper graph in $S$} is a one-complex $G$ minus some subset of the
vertices, properly embedded in $S$.  
A connected proper graph is \emph{essential} if it is not properly homotopic into an end of $S$
 or to a point. In general, $G$ is essential if some component is essential.

A proper arc or curve $a$ is \emph{nearly simple} in a proper graph $G$ if $a$ is properly homotopic to a proper path or curve in $G$ which visits no vertex of $G$ more than twice.
Note that a proper graph $G$ in $S$ is essential if and only if it carries an essential arc or curve. Define
\begin{equation}\label{A_S definition}
\A_S(G) = \{a \in \A(S)^{(0)} : a \text{ is nearly simple in } G\}.
\end{equation}

\begin{corollary} \label{cor:diam_bound}
Suppose that $G$ is an essential proper graph in $S$ with at most $2|\chi(S)| +1$ vertices. Then $\diam_S(\A_S(G)) \le D$ for $D = 15$.
\end{corollary}

\begin{proof}
Let $a$ and $b$ be essential arcs or curves that are nearly simple in $G$. Realize $a$ and $b$ in a small neighborhood of $G$ so that they intersect only in neighborhoods of the vertices of $G$. Since $a$ and $b$ each pass through any neighborhood of a vertex at most twice, they intersect at most $4(2|\chi(S)| +1)$ times. By the computations in \cref{eqn:i bound 8}, this implies that $d_S(a,b) \le 15$.
\end{proof}

\subsection{Veering triangulations}
\label{veering defs}

Our basic object here is a Riemann surface $X$ with an integrable holomorphic quadratic
differential $q$, which fits into a sequence
$$
\cX \subset X \subset \overline X
$$
as follows: $\overline X$ is a closed Riemann surface on which $q$ extends to a meromorphic
quadratic differential, and $\PP = \overline X \ssm X$ is a finite set of
{\em punctures} which includes the poles of $q$, if any.
Let $\sing(q)$ be the union of $\PP$ with the zeros of $q$, so that
$$
\mathrm{poles}(q) \subseteq \PP \subseteq \sing(q)
$$
and set $\cX = \overline X \ssm \sing(q)$.
When $X=\cX$ we say that $X$ is {\em fully-punctured}. 

Let $\lambda^+$ and $\lambda^-$ be the vertical and horizontal
foliations of $q$, which we assume contain no saddle connections. 

The constructions of Agol \cite{agol2011ideal} and Gueritaud \cite{gueritaud}
yield a fibration
$$
\Pi \colon \N \to \cX
$$
whose fibers are oriented lines, so that $\N \cong \cX\times\R$, and $\N$ is equipped
with an ideal triangulation $\tau$ whose tetrahedra, called $\tau$-simplices, are 
characterized by the following description:

Let $p\colon\til \cX\to \cX$ be the universal covering map and $\widehat X$ the metric completion
of $\til \cX$. Note
that $p$ extends to an infinitely branched covering $\widehat X \to \overline
X$. 

\realfig{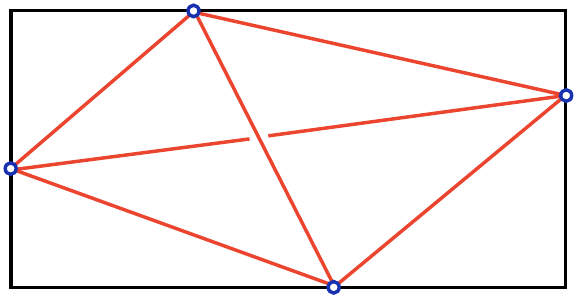}{A maximal singularity-free rectangle $R$ defines a
  tetrahedron equipped with a map into $R$.}

A {\em singularity-free rectangle} in $\widehat X$ is an embedded rectangle whose
edges are leaf segments of the lifts of $\lambda^\pm$ and whose
interior contains no singularities of $\widehat X$. 
If $R$ is a {\em maximal} singularity-free rectangle in $\widehat X$ then
it contains exactly one singularity on the interior of  each edge. The four singularities span a
quadrilateral in $R$ which we can think of as the image of a tetrahedron by a projection
map whose fibers are  intervals, as pictured in \Cref{gue-tetra}.

The tetrahedra of $\N$ are identified with all such tetrahedra, up to the action of
$\pi_1(\cX)$, where the restriction of $\Pi$ is exactly this projection to the
rectangles, followed by $p$. 

A \define{$\tau$-edge} in $\cX$  is the $\Pi$-image of an edge of $\tau$, or equivalently a
saddle connection of $q$ whose lift to $\til \cX$ spans a singularity-free rectangle. A
$\tau$-edge in $X$ or $\overline X$  is the closure of a $\tau$-edge in $\cX$. 

\medskip

When $\phi \colon S\to S$ is a pseudo-Anosov homeomorphism, let $(X,q)$ denote $S$ endowed with a
quadratic-differential $q$ whose foliations are the stable and unstable foliations of
$\phi$. Then $\lambda^\pm$ have no saddle connections so we may construct $\N$, on which $\phi$
induces a simplicial homeomorphism $\Phi$ of $\N$, whose quotient $\cM_\phi$ is the mapping 
torus of $\phi|_{\cS}$.
Equivalently, $\cM = \cM_\phi$ is obtained from the mapping torus $M_\phi$ by removing the
singular orbits of its suspension flow, which we discuss next.

\subsection{Fibered faces of the Thurston norm}
\label{sec:tnorm}
Let $M$ be a finite-volume hyperbolic $3$-manifold.
A fibration $\sigma\colon M\to S^1$ of $M$ over the circle 
comes with the following structure:
there is a primitve integral cohomology class in $H^1(M;\Z)$
represented by $\sigma_* \colon \pi_1M\to \Z$, which is the Poincar\'e dual
of the fiber $F$. There is also a 
representation of $M$ as a quotient $\left (F\times\R \right) /\Phi$
where $\Phi(x,t) = (\phi(x),t-1)$ and where $\phi \colon F\to F$ is a pseudo-Anosov
homeomorphism called the monodromy map. 
The map $\phi$ has stable and unstable (singular) measured foliations
$\lambda^+$ and $\lambda^-$ on $F$. Finally there is the suspension
flow inherited from the natural $\R$ action on $F\times\R$, and 
suspensions $\Lambda^\pm$ of $\lambda^\pm$ which are flow-invariant 2-dimensional
foliations of $M$. 
Note that the deck transformation $\Phi$ translates in the \emph{opposite} direction of the lifted flow. This is so that the first return map to the fiber $F$ equals $\phi$.

The fibrations of $M$ are organized by the {\em Thurston norm}
$||\cdot||$ on $H^1(M;\R)$ \cite{thurston1986norm} (see also \cite{fried1979fibrations}). This norm has a
polyhedral unit ball $B$ with the following properties:
\begin{enumerate}
  \item Every cohomology class dual to a fiber is in the
    cone $\R_+\FF$ over a top-dimensional open face $\FF$ of $B$.

  \item If $\R_+\FF$ contains a cohomology class dual to a fiber
      then {\em every} primitive integral class in $\R_+\FF$ is dual to a
      fiber. $\FF$ is called a {\em fibered face} and its primitive integral
      classes are called fibered classes.

  \item For a fibered class $\omega$ with associated fiber $F$, 
         $||\omega||=-\chi(F)$.
\end{enumerate}
In particular if $\dim H^1(M;\R)\ge 2$ and $M$ is fibered then there
are infinitely many fibrations, with fibers of arbitrarily large
complexity.
We will abuse terminology by saying that a fiber (rather than
its Poincar\'e dual) is in $\R_+\FF$. 

The fibered faces also organize the suspension flows and the
stable/unstable foliations: If $\FF$ is a fibered face then there is a
single flow $\psi$ and a single pair $\Lambda^\pm$ of foliations whose leaves
are invariant by $\psi$, such that {\em every} fibration 
associated to $\R_+\FF$ may be isotoped so that its suspension flow is
$\psi$ up to a reparameterization, and the foliations $\lambda^\pm$ for the monodromy of its fiber $F$ are
$\Lambda^\pm\intersect F$.
These results were proven by Fried \cite{fried1982geometry}; see also
McMullen \cite{mcmullen2000polynomial}.  

%

Finally, we note that  the veering triangulation of $\cM$, like the flow itself, is an invariant of the fibered face containing the fiber $\cS$ (see \cite{agol-overflow} or \cite[Proposition 2.7]{veering1}).

\subsection{Subsurfaces, $q$--compatibility, and $\tau$--compatibility} \label{sec:compatibility}
We conclude this section by reviewing some essential constructions from \cite{veering1} and direct the reader there for the full details. In short, the idea is that if $Y$ is a subsurface of $(X,q)$ with $d_Y(\lambda^-,\lambda^+)$ sufficiently large, then $Y$ has particularly nice forms; the first with respect to the $q$-metric, and the second with respect to $\tau$. 

Let $Y \subset X$ be an essential compact subsurface, and let
$X_Y=\til X/\pi_1(Y)$ be the associated cover of $X$. 
We say a boundary component of $Y$ is {\em puncture-parallel} if it bounds a
disk in $\overline X \ssm Y$ that contains a single point of $\PP$. We
denote the corresponding subset of $\PP$ by $\PP_Y$ and refer to them
as the \emph{punctures} of $Y$. Let $\til{\PP}_Y$ denote the subset of
punctures of $X_Y$ which are encircled by the boundary components of
the lift of $Y$ to $X_Y$. 
In terms of the completed space $\overline X_Y$, $\til \PP_Y$ is exactly the set of completion points which have finite total angle.  
Let  $\boundary_0Y$ denote the union of the puncture-parallel components of
$\boundary Y$ and let $\boundary'Y$ denote the rest. Observe that the 
components of $\boundary_0 Y$ are in natural bijection with $\PP_Y$ and set
$Y' = Y\ssm\boundary_0Y$. 

Identifying $\til X$ with $\HH^2$, 
let $\Lambda\subset \boundary\HH^2$ be the limit set of
$\pi_1(Y)$, $\Omega = 
\boundary\HH^2 \ssm \Lambda$, and $\widehat\PP_Y\subset \Lambda$
the set of parabolic fixed points of $\pi_1(Y)$.
Let $C(X_Y)$ denote the compactification of $X_Y$ given by
$(\HH^2 \union \Omega\union \widehat\PP_Y)/\pi_1(Y)$, adding a point for each 
puncture-parallel end of $X_Y$, and a circle for each of the other
ends. 

%
%
%
%

\subsubsection*{$q$-convex hulls}  

As above, identify $\til X$ with $\HH^2$. Let 
$\Lambda \subset \partial \HH^2$ be a closed set and let 
$\CH(\Lambda)$ be the convex hull of $\Lambda$ in $\HH^2$. Using the results of
\cite[Section 2.3]{veering1}, we define the $q$-convex hull
$\CH_q(\Lambda)$ as follows.

Assume first that $\Lambda$ has at least 3 points.
Each boundary geodesic $l$ of $\CH(\Lambda)$ has the same
endpoints as a
(biinfinite) $q$-geodesic $l_q$ in $\widehat X$ (we note that $l_q$ may meet $\partial
\HH^2$ at interior points).
Further, $l_q$ is
unique unless it is part of a parallel family of geodesics, making a
Euclidean strip.

$\widehat X$ is divided by $l_q$ into two sides, 
and one of the sides, which we call  
$\side{l}$, meets $\partial \HH^2$
 in a subset of the complement of $\Lambda$.
The side $\side{l}$ is either a disk
or a string of disks attached along completion
points.
If $l_q$ is one of 
a parallel family of geodesics, we include this family in $\side{l}$.
After deleting from $\widehat X$ the interiors of $\side{l}$ for all $l$ 
in $\boundary \CH(\Lambda)$,
we obtain $\CH_q(\Lambda)$, the
$q$-convex hull. 

If $\Lambda$ has 2 points then $\CH_q(\Lambda)$ is the closed Euclidean strip
formed by the union of $q$-geodesics joining those two points.

\medskip

Now fixing a subsurface $Y$ we can define a $q$-convex hull for the
cover $X_Y$ by taking a quotient of the $q$-convex hull $\CH_q(\Lambda_Y)$
 of the limit
set $\Lambda_Y$ of $\pi_1(Y)$. This quotient, which we will denote by
$\CH_q(X_Y)$, lies in the completion $\overline{X}_Y$.
We remark that in general  $\CH_q(X_Y)$ may be a total mess, e.g. it may have empty interior.

%
%

\subsubsection*{$q$--compatibility}
Let $\widehat\iota:Y \to X_Y$ be the lift of the inclusion map to the
cover.
We say that the subsurface $Y$ of $X$ is \define{$q$-compatible} if the interior of $\CH_q(\Lambda_Y)$ is a disk. In this case, \cite[Lemma 2.6]{veering1} implies that
$\widehat \iota \colon Y \to X_Y$ is homotopic to a map $\widehat \iota_q \colon Y \to \overline X_Y$ which restricts to a homeomorphism from $Y \ssm \partial_0 Y$ to 
\begin{align}
Y_q =\CH_q(X_Y) \ssm \til \PP_Y
\end{align}
 
 \realfig{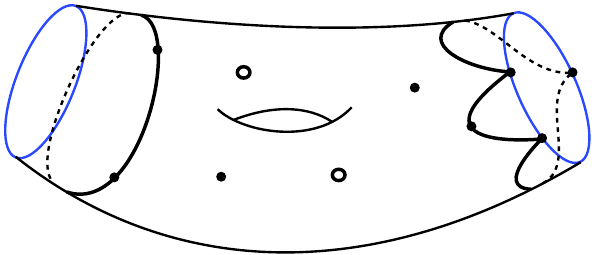}{The image of a $q$-compatible subsurface $Y$ in $\overline X_Y$ under $\widehat \iota_q$. Open circles are points of $\til \PP_Y$ (corresponding to the image of $\partial_0 Y$) and dots are singularities not contained in $\til \PP_Y$. The ideal boundary of $X_Y$ is in blue.}

We recall that by \cite[Lemma 5.1]{veering1}, if $Y$ is $q$-compatible, then 
  \begin{enumerate}
    \item the projection $\iota_q \colon Y \to \overline X$ of $\widehat
      \iota_q$ to $\overline X$ is an embedding from $\int(Y)$ into $X$
      which is homotopic to the inclusion, and
      \item $\widehat\iota_q(\boundary Y \ssm \boundary_0Y )$ does not pass through points of $\til\PP_Y$.
  \end{enumerate}
See \Cref{Y_q_cartoon_2}. The embedded image $\iota_q(\int(Y))$ is an open representative of $Y$ in $X$ and is denoted $\int_q(Y)$.

The following is our main tool for proving $q$-compatibility; it is \cite[Proposition 5.2]{veering1}.

\begin{proposition}[$q$-Compatibility] \label{prop: q_compatible}
Let $Y\subset X$ be an essential subsurface. 
\begin{enumerate}
\item If $Y$ is nonannular and $d_Y(\lambda^-,\lambda^+) \ge 3$,  then $Y$ is $q$-compatible. 
\item If $Y$ is an annulus and $d_Y(\lambda^-,\lambda^+) \ge 4$, then
$Y$ is $q$-compatible. In this case, $\int_q(Y)$ is a flat cylinder. 
\end{enumerate}
\end{proposition}

We remark that the constants in \cite[Proposition 5.2]{veering1} are slightly different since there $d_Y$ was defined to be the minimal distance between projections.

\subsubsection*{$\tau$--compatibility} \label{sec:tau_hull}
Next we focus on compatibility with respect to the veering triangulation.
Call a $q$-compatible subsurface $Y \subset X$ \define{$\tau$-compatible} if the map $\widehat\iota_q
\colon Y \to \overline X_Y$ is homotopic rel $\partial_0 Y$ to a map
$\widehat\iota_\tau:Y\to \overline X_Y$ which is an embedding on $Y' = Y\ssm \partial_0Y$ such that
\begin{enumerate}
\item $\widehat\iota_\tau$ takes each component of  $\boundary'Y = \partial Y \ssm \partial_0 Y$ to a simple curve in $\overline X_Y \ssm \til{\PP}_Y$ composed of
a union of $\tau$-edges and 
  \item the map $\iota_\tau \colon Y \to \overline X$ obtained by composing
    $\widehat \iota_\tau$ with $\overline X_Y \to \overline X$ restricts to an embedding from $\int(Y)$ into $X$.
\end{enumerate}

When the subsurface $Y$ is $\tau$-compatible, we set 
\begin{align}
\partial_\tau Y \equiv \iota_\tau(\partial'Y)
\end{align}
which is a collection of $\tau$-edges with disjoint interiors. We call $\partial_\tau Y$
the \emph{$\tau$--boundary} of $Y$ and consider it as a $1$-complex of $\tau$-edges in
$X$. (At times we will also think of $\partial_\tau Y$ as a collection of disjoint $\tau$-edges in the fully-punctured surface $\cX$.) Similar to the situation of a $q$-compatible subsurface, if $Y$ is $\tau$-compatible then one component of $X \ssm \partial_\tau Y$ is an open subsurface isotopic to the interior of $Y$; this is the image $\iota_\tau(\int(Y))$ and is denoted $\int_\tau (Y)$. 
For future reference, we set $Y_\tau \subset X_Y$ to be the intersection of $X_Y$ with the image of $\widehat \iota_\tau$. 
By definition, the covering $X_Y \to X$ maps the interior of $Y_\tau$ homeomorphically onto $\int_\tau(Y)$.

The following result is Theorem 5.3 of \cite{veering1}.

\begin{theorem}[$\tau$-Compatibility]\label{thm: tau-compatible}
Let $Y\subset X$ be an essential subsurface. 
\begin{enumerate}
\item If $Y$ is nonannular and $d_Y(\lambda^-,\lambda^+) \ge 3$, 
 then $Y$ is $\tau$-compatible.
 \item If $Y$ is an annulus and $d_Y(\lambda^-,\lambda^+) \ge 4$, then
$Y$ is $\tau$-compatible. 
\end{enumerate}
\end{theorem}

The comment following \Cref{prop: q_compatible} also applies here.

%% file: veering.tex

\section{Veering triangulations and subsurfaces}
\label{sec:veering}

Here we study the connection between sections of the bundle $\N \to \cX$ and projections to subsurfaces of $X$. In brief, we tailor the theory of subsurface projections to the veering structure. This is accomplished in \Cref{lem:overlap_tau} and \Cref{prop:proj_close}.

\subsection{Sections of the veering triangulation}\label{sec:sections}
A \define{section} of the veering triangulation $\tau$  in $\N$ is a simplicial
embedding $s \colon (\cX,\Delta) \to \N$ 
that is a section of the
fibration $\Pi \colon \N \to \cX$. 
Here, $\Delta$ is an ideal triangulation of $\cX$, which by construction consists of $\tau$--edges.
We will also refer to the image of $s$ in $\N$, which 
we often denote by $T$, as a section. 

There is a bijective correspondence between sections of $\N$ and
ideal triangulations of $\cX$ by $\tau$--edges. More generally, 
we use the notation $\Pi_*$ to denote the map that associates to any
subcomplex of a section the corresponding union of $\tau$--simplices
of $\cX$, and we use $\Pi^*$ to denote its inverse. In particular, if
$K$ is a union of disjoint $\tau$--edges of $\cX$, $\Pi^*(K)$ is the
subcomplex of $\N$ obtained by lifting its simplices to $\N$. 
Note that $T$ and $T'$ differ by a tetrahedron move in $\N$ if and only if the ideal triangulations $\Pi_*(T)$ and $\Pi_*(T')$ differ by a diagonal exchange. Here, an upward (downward) tetrahedron move on a section $T$ replaces two adjacent faces at the bottom (top) of a tetrahedron with the two adjacent top (bottom) faces.


Since the fibers of $\Pi \colon \N \to \cX$ give $\N$ an oriented foliation by lines and each of these lines meets each section exactly once, we have the following observation: For each $x \in \N$ and each section $T$ of $\N$, it makes sense to write $x \le T$ or $x \ge T$ depending on whether $x$ lies weakly below or above $T$ along the orientation of the line through $x$. (Here, $x \le T$ and $x \ge T$ imply that $x \in T$.) In fact, this ordering extends to each simplex of $\N$; we write $\sigma \le T$ if $x \le T$ for each $x \in \sigma$. Since we will use this fiberwise ordering for several (simplicial) constructions, it is important to note that it is consistent along simplices; that is,
if $x \le T$ and $\sigma$ is the smallest simplex containing $x$, then $\sigma \le T$.
Finally, if $K$ is a subcomplex of $\N$, then $K \le T$ if for each simplex $\sigma$ of $K$, $\sigma \le T$.

In \cite[Section 2.1]{veering1} we define a strict partial order among $\tau$-edges using their
spanning rectangles: if $e$ crosses $f$ we say that $e > f$ if $e$
crosses the spanning rectangle of $f$ from top to bottom, and $f$
crosses the spanning rectangle of $e$ from left to right (i.e. if the 
slope of $e$ is greater than the slope of $f$). A priori, this partial order is defined in the universal cover $\wt X$, but it projects consistently to $\cX$ and so defines a partial order of $\tau$-edges there as well.
See \Cref{fig:greater}.

\begin{figure}[htbp]
\begin{center}
\includegraphics[]{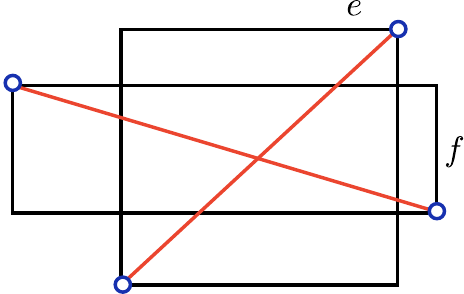}
\caption{Two $\tau$-edges with $e>f$.}
\label{fig:greater}
\end{center}
\end{figure}

This definition is
consistent with the ordering of the simplices $\Pi^*(e)$, $\Pi^*(f)$  in $\N$, and in particular
\begin{lemma} \label{lem:up_is_up}
$K \le T$ if and only if whenever an edge $e$ of $\Pi_*(T)$ crosses an edge $f$ of $\Pi_*(K)$, we have $f < e$.
\end{lemma}

\begin{proof}
First, let $f$ be any $\tau$-edge in $\cX$ and let $\Delta$ be a triangulation of $\cX$ by $\tau$-edges. By \cite[Lemma 3.4]{veering1}, if $e$ is an edge of $\Delta$ and $e>f$, then there is an edge of $\Delta$ crossing $f$ which is downward flippable, meaning that there is a diagonal exchange of $\Delta$ replacing it with an edge of smaller slope. (An analogous statement holds if $f>e$.) Such a diagonal exchange results in a triangulation $\Delta'$ either containing $f$ or still containing an edge $e'$ with $e' >f$. After finitely many downward diagonal exchanges, we arrive at a triangulation by $\tau$-edges which contains the edge $f$. (See \cite[Section 3]{veering1} for details.)
Translating this statement to $\N$, this means that starting with $\Pi^*(\Delta)$ there is a sequence of downward tetrahedron moves resulting in a section containing $\Pi^*(f)$. Hence, $\Pi^*(f) \le \Pi^*(\Delta)$

This, together with the corresponding result for when $K\ge T$, implies the lemma.
\end{proof}

Given sections $T_1$ and $T_2$, we use $U(T_1,T_2)$ to denote the subcomplex of $\N$ between them. Formally, $U(T_1,T_2)$ is the subcomplex of $\N$ which is the union of all simplices $\sigma$ such that either $T_1 \le \sigma \le T_2$ or $T_2 \le \sigma \le T_1$. 

It will be helpful to consider the lattice structure of sections. For sections $T_1,T_2$, we denote their fiberwise \define{maximum} by $T_1 \vmax T_2$. If we name the oriented fiber containing $x$ by $l_x$, this is the subset $\{x \in \N : x = \max_l \{l_x \cap T_1, l_x\cap T_2\}\}$, where the max is taken with respect to the ordering on each $l_x$.
\begin{lemma}
$T_1 \vmax T_2$ is a section.
\end{lemma}
\begin{proof}
Since the restriction of $\Pi$ to $T_1 \vmax T_2$ is a homeomorphism to $\cX$, it suffices to show that $T_1 \vmax T_2$ is a subcomplex of $\N$. Let $x \in T_1 \vmax T_2$ and suppose that $x$ is contained in $T_1$. Then $x \ge T_2$ and so if $\sigma$ is the minimal simplex of $T_1$ containing $x$, we see $\sigma \ge T_2$. Hence, $\sigma \subset T_1 \vmax T_2$ and we conclude that $T_1 \vmax T_2$ is indeed a section.
\end{proof}

We can define the \define{minimum} $T_1 \vmin T_2$ of two sections similarly. With this terminology, it makes sense to say that $T_1 \vmax T_2$ is the \define{top} of $U(T_1,T_2)$. More precisely, $T_1 \vmax T_2 \subset U(T_1,T_2)$ and for every simplex $\sigma \subset U(T_1,T_2)$, $\sigma \le T_1 \vmax T_2$. 
Similarly, we say that $T_1 \vmin T_2$ is the \define{bottom} of $U(T_1,T_2)$. Further, using our definitions we see that
\[
U(T_1,T_2)  = U(T_1 \vmin T_2, T_1 \vmax T_2).
\]
Note that the part of $T_1$ that lies above $T_2$ is $T_1 \intersect (T_1 \vmax T_2)$. 


\subsubsection*{Sections through a subcomplex}
Let $E \subset \cX$ be a union of disjoint $\tau$--edges and set $K = \Pi^*(E)$ to be the corresponding subcomplex of $\N$. (Our primary example will be $E =\partial_\tau Y$ for a $\tau$-compatible subsurface $Y$ of $X$. In this situation we think of $\partial_\tau Y$ as a collection of $\tau$-edges of $\cX$.)
We define $T(E) = T(K)$ to be the set of sections of $\N$ which contain $K$ as a subcomplex. Similarly, we define $\Delta(E) = \Delta(K)$ as the set of ideal triangulations of $\cX$ by $\tau$--edges containing $E$.
We recall the following two basic results from \cite{veering1}. 
The first states simply that $T(E)$ is nonempty. It is  \cite[Lemma 3.2]{veering1}.

\begin{lemma}[Extension lemma] \label{lem:extension}
Suppose that $E$ is a collection of
$\tau$-edges in $\cX$  with
pairwise disjoint interiors. Then $T(E)$ is nonempty.
\end{lemma}

The second (\cite[Proposition 3.3]{veering1}) states that $T(K)$ is always connected by tetrahedron
moves. This includes in particular the case of $T(\emptyset)$, the set
of all sections.

\begin{lemma}[Connectivity] \label{lem:connectivity}
  If $E$ is a collection of
$\tau$-edges in $X$  with
pairwise disjoint interiors,
  then $\Delta(E)$ is connected
  via diagonal exchanges. 
  In terms of $\N$, for $K = \Pi^*(E)$,
  $T(K)$ is connected via tetrahedron moves.
  
  Moreover, if $T_1,T_2 \in T(K)$ with $T_1 \le T_2$, then there is a sequence 
  of \emph{upward} tetrahedron moves from $T_1$ to $T_2$ through sections of $T(K)$. 
\end{lemma}

As explained in \cite[Corollary 3.6]{veering1}, whenever $E \neq \emptyset$, there is a well-defined \define{top} $T^+$ and \define{bottom} $T^-$ of $T(E)$. That is, $T^+ \in T(E)$ and for any $T \in T(E)$, $T \le T^+$.

\vspace{5mm}

\noindent \define{$\phi$-sections.}
Suppose that $q$ is a quadratic differential associated to a pseudo-Anosov homeomorphism $\phi$. Recall that the deck transformation $\Phi$ of $\N$ is chosen to translate in the opposite direction of the flow.

 Say that a section $T$ of the veering triangulation $\tau$ is a
 \emph{$\phi$--section} if $\Phi(T) \le T$.
In other words, $T$ is a $\phi$--section if every $\tau$-edge of $\Pi_*(\Phi(T)) = \phi(\Pi_*(T))$ which crosses a $\tau$-edge of $\Pi_*(T)$ does so with lesser slope. Note that if $T$ is a $\phi$--section, then $\Phi^j(T) \le \Phi^i(T)$ for all $i \le j$.

Agol's original construction produces a veering triangulation from a sequence of diagonal exchanges through $\phi$-sections \cite[Proposition 4.2]{agol2011ideal}. In fact, he proves 

\begin{lemma}[Agol] \label{lem:AG_sec}
There is a sweep-out of $\tau$ through $\phi$--sections. That is,
there is a sequence $(T_i)_{i \in \mathbb{Z}}$ of $\phi$--sections
such that $T_{i+1}$ is obtained from $T_i$ by simultaneous upward
tetrahedron moves.
\end{lemma}

\subsection{Projections to $\tau$--compatible subsurfaces} \label{sec:proj_tau}

In this section we define a variant of the subsurface projections
$\pi_Y$ that is adapted to the 
simplicial structure of $\tau$. In the
hyperbolic setting, $\pi_Y(\alpha)$ can be defined using the geodesic
representatives of the surface $Y$ and the curve $\alpha$. In our
setting we need to use the simplicial representative $Y_\tau$ 
of a
$\tau$-compatible surface $Y$ and
a collection of $\tau$-edges representing $\alpha$.
The main result here will be \Cref{prop:proj_close}, which
shows, in a suitable setting, that the simplicial variant of the
projection is uniformly close to the usual notion.

Recall first the notion of a proper graph $G$ in a surface from
\Cref{sec:prop_graph}
 and its image $\A_S(G)$ in
the arc graph (Definition \ref{A_S definition}).

If $E$ is a collection of $\tau$--edges of $\cX$ with disjoint
interiors, then its closure in $X$, $\cl_X(E)$, is a proper
graph. This is the union of the corresponding saddle connections in
$X$.
In particular if $K$ is a subcomplex of a section then
$E=\Pi_*(K^{(1)})$ is such a collection of $\tau$-edges and we make
the notational abbreviation
\begin{equation}\label{A_S for section}
\A_X(K) = \A_X(\cl_X(\Pi_*(K^{(1)}))).
\end{equation}
Note that for any section $T$, we have (\Cref{cor:diam_bound}) $\diam(\A_X(T)) \le D$, where $D= 15$. 


Suppose now that $Y \subset X$ is a $\tau$--compatible nonannular subsurface and
$G \subset X$ is a union of $\tau$--edges with disjoint interiors (i.e. \emph{a proper graph of $\tau$-edges}).  Then
$\int_\tau(Y) \cap G$ is a proper graph in $\int_\tau (Y)$
and we set 
\begin{equation}\label{pi Y tau def}
\pi_Y^\tau (G) = \A_Y(\int_\tau(Y) \cap G). 
\end{equation}
Note that this could in general be empty, if $\int_\tau(Y) \cap G$ is
not essential. 

When $Y$ is a $\tau$-compatible annulus, $\int_\tau(Y) \cap G$ is a collection of disjoint arcs each of which is contained in the interior of a $\tau$-edge. Taking the preimages of these $\tau$-edges in $X_Y$,
we obtain the projection $\pi_Y^\tau (G)$ by associating to each such $\tau$-edge $a$
that joins \emph{opposite} sides of $\partial Y_\tau \subset X_Y$
the collection of complete $q$-geodesics $a^*$ in $X_Y$ that contain it. Each of these geodesics gives a well-defined arc of $\A(Y) = \A(X_Y)$, and if here are no such $\tau$-edges, then the projection is empty.


In general, this notion of subsurface projection is easily extended to a
subcomplex $K$ of a section of $\N$, in analogy with (\ref{A_S for
  section}). We simply write:
\begin{equation}
  \pi_Y^\tau(K)  = \pi_Y^\tau(\cl_X(\Pi_*(K^{(1)}))).
\end{equation}

Finally, we define the subsurface distance between subcomplexes $K_1$ and $K_2$ to be
\begin{equation}\label{dY complexes}
d_Y(K_1,K_2) = d_Y(\pi^\tau_Y(K_1), \pi^\tau_Y(K_2)).
\end{equation}

\medskip


The following lemma establishes some important technical properties of $\tau$-compatible subsurfaces. Key to the argument is the construction of $\int_\tau(Y)$ from $\int_q(Y)$ that appears in \cite[Theorem 5.3]{veering1} and is illustrated in \Cref{fig:thull-twice}.

\begin{figure}[htbp]
\begin{center}
\includegraphics[width=.7 \textwidth]{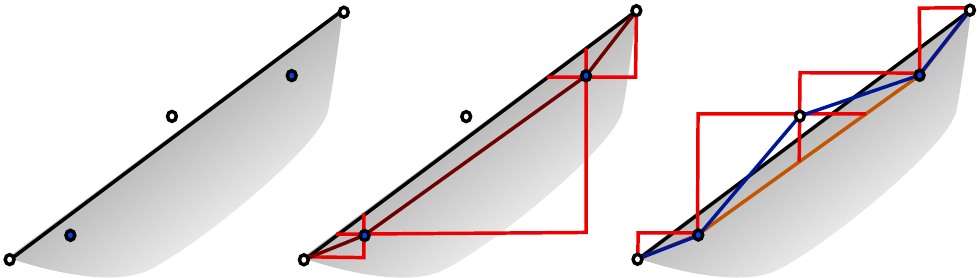}
\caption{Left: $\int_q(Y)$ is shaded, contains the blue singularities, and its boundary contains the black saddle connection. Middle: the `inner {\bf t}-hull construction' isotopes the surface within itself; its new boundary consists of saddle connections (dark red) through the blue singularities. Right: $\int_\tau(Y)$, whose boundary consists of $\tau$-edges (blue), is then produced with the `outer {\bf t}-hull construction.' }
\label{fig:thull-twice}
\end{center}
\end{figure}

We remark that one difficulty in what follows is that the $\tau$-representative $\int_\tau(Y)$, unlike the $q$-representative $\int_q(Y)$, is \emph{not} convex with respect to the $q$ metric. 

\begin{lemma} \label{lem:overlap_tau}
Let $Y$ and $Z$ be $\tau$-compatible subsurfaces of $X$ and let $G \subset X$ be
a proper graph of $\tau$-edges.
\begin{enumerate}
\item The diameter $\diam_Y(\pi_Y^\tau (G))$ is bounded by $D = 15$. If $Y$ is an annulus, then $\diam_Y(\pi_Y^\tau (G)) \le 3$.
\item If $Y$ and $Z$ are disjoint then so are $\int_\tau(Y)$ and $\int_\tau (Z)$.
\item The subsurface $\int_\tau(Y)$ is in minimal position with the foliation $\lambda^\pm$. In particular, the arcs of $\int_\tau(Y) \cap \lambda^\pm$ agree with the arcs of $\pi_Y(\lambda^\pm)$.
\end{enumerate}
\end{lemma}



\begin{proof}
The graph $G \cap \int_\tau(Y)$ has its vertices in $\int_\tau(Y) \cap \sing(q)$. By Gauss--Bonnet, $|\int_\tau(Y) \cap \sing(q)| \le 2|\chi(Y)|$, and so $(1)$ follows from \Cref{cor:diam_bound} when $Y$ is not an annulus.
 
If $Y$ is an annulus, then recall that 
$\int_q(Y)$ is a flat annulus and that $\int_q(Y) \subset \int_\tau(Y)$ since the process of going from the $q$-hull to the $\tau$-hull for annuli only pushes outward (c.f. 
\cite[Remark 5.4]{veering1}). Lifting to the annular cover $X_Y$, let $a,b$ be $\tau$-edge coming from $G$ that join opposite sides of $Y_\tau$. Then $a$ and $b$ are disjoint and any of their $q$-geodesic extensions $a^*, b^*$ cross the maximal open flat annulus $\int_q(Y)$ of $X_Y$ in subsegments of $a,b$, respectively. Moreover, by a standard Gauss--Bonnet argument (e.g. \cite[Lemma 3.8]{rafi2005characterization}), any two $q$-geodesic segments intersect at most once in any component of $X_Y \ssm \int_q(Y)$. Hence, $a^*,b^*$ intersect at most twice and so $\diam_Y(\pi_Y^\tau (G)) \le 3$ when $Y$ is an annulus.

Items $(2)$ and $(3)$ follow exactly as in \cite[Lemma 6.1]{veering1}. As that lemma was proven only in the fully-punctured case, we note that in general for item $(2)$
one must perform the inner $\thull$-hull construction 
(middle of \Cref{fig:thull-twice}) as an intermediate step. 
However, since this pushes each surface within itself, it must preserve disjointness. For $(3)$, the isotopy from $\int_q(Y)$ to $\int_\tau(Y)$ pushes along the leaves of $\lambda^+$ (or $\lambda^-$). Hence, the leaves of $\lambda^+$ (or $\lambda^-$) are in minimal position with $\int_\tau(Y)$ since they are with $\int_q(Y)$, by the local CAT$(0)$ geometry.
\end{proof}

We note that it follows from \Cref{lem:overlap_tau}.(3) (or directly from its proof) that if $\til \lambda^\pm$ are the lifts of $\lambda^\pm$ to $\til X$ and $C^\tau_Y$ is a component of the preimage of $\int_\tau(Y)$, then the intersection of  $C^\tau_Y$ with each leaf of $\til \lambda^\pm$ is connected.

The following proposition relates the projections defined here to the usual notion of subsurface projection.

\begin{proposition} \label{prop:proj_close}
Let $Y,Z$ be $\tau$--compatible subsurfaces of $X$ and suppose that $\pi_Z(\partial Y) \neq \emptyset$. Further assume that $d_Y(\lambda^-,\lambda^+) \ge 3$ (or if $Y$ is an annulus that $d_Y(\lambda^-,\lambda^+) \ge 6$).
Then
\[
\diam_{Z} \big(\pi_Z^\tau(\partial_\tau Y) \cup \pi_Z(\partial Y) \big) \le 7.
\]
\end{proposition}



As usual, let $X_Z$ denote the $Z$--cover of $X$. In this subsection, a \emph{core} $Z'$ of $X_Z$ is a submanifold with boundary which is a complementary component 
 of simple curves and proper arcs such that $Z' \to X_Z$ is a homotopy equivalence. This definition includes the usual convex core in the hyperbolic metric as well as $Z_\tau \subset X_Z$ (\Cref{sec:tau_hull}).
In general, $\partial Z' \subset X_Z$ is a collection of curves and arcs.

\begin{figure}[htbp]
\begin{center}
\includegraphics[width=.7 \textwidth]{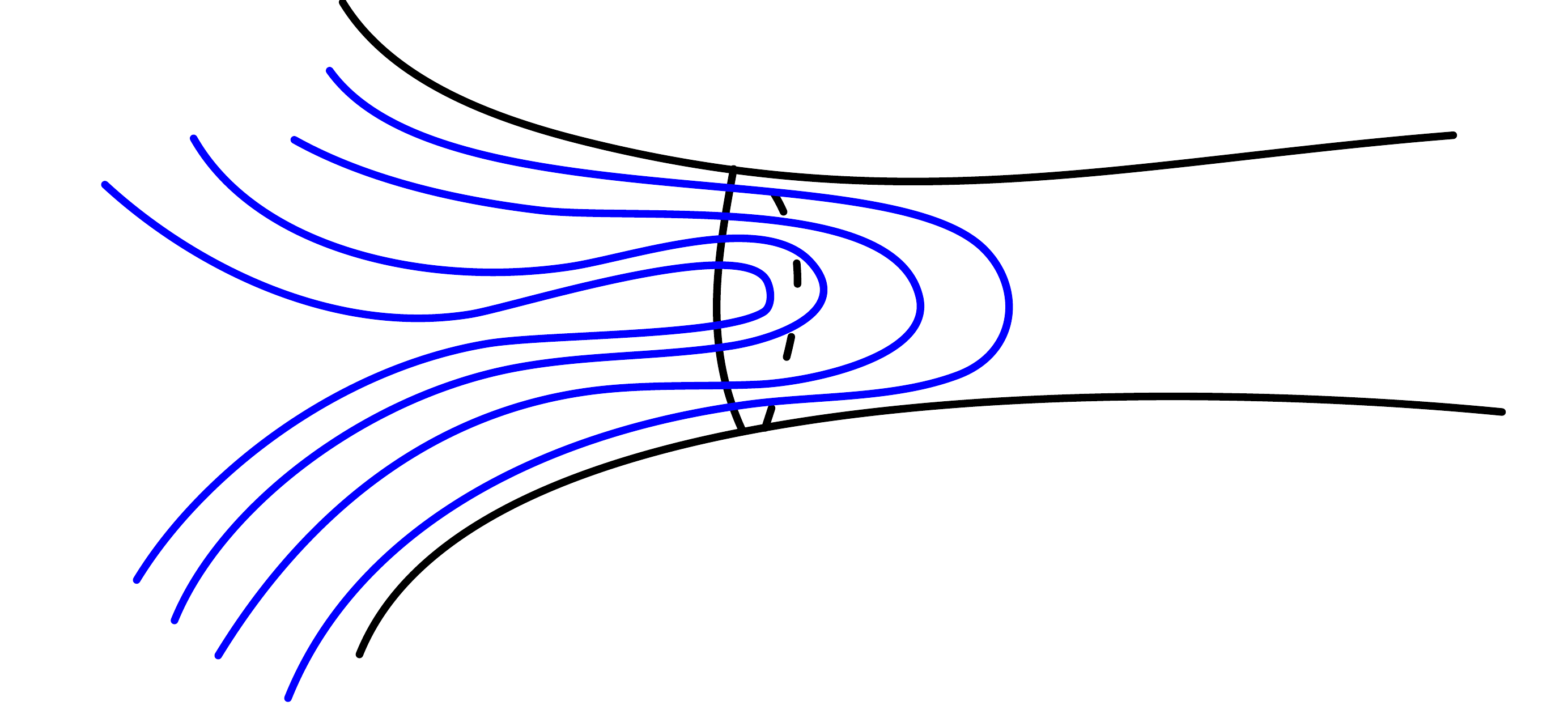}
\caption{The curve/arc $\gamma$ (blue) is in $4$-position with respect to $\partial Z$, but not $3$--position. Note that the arcs of $\partial Z$ cobounding these bigons are nested.}
\label{fig:k_pos}
\end{center}
\end{figure}
 
Let $\gamma$ be any essential curve or proper arc in $X_Z$. Note that $\gamma$ may not be in minimal position with $\partial Z'$, that is, there may be bigons between arcs of $\gamma$ and $\partial Z'$. To handle this situation, we make the following definition:
For some $k\ge0$, we say that $\gamma$ is in \define{$k$--position} with respect to $\partial Z'$ if $k$ is the largest integer such that a collection of $k$ \emph{nested} subsegments of $\partial Z'$ cobound bigons with subsegments of $\gamma$ whose interiors are contained in $X_Z \ssm Z'$. 
See \Cref{fig:k_pos}.

The point of this condition is the following lemma.

\begin{lemma} \label{lem:no_nest}
Suppose that $Z$ is a nonannular subsurface of $X$ and that $X_Z$ and $Z'$ are as above. Let $\gamma$ be an essential curve or proper arc in $X_Z$
which is in $k$--position with respect to $\partial Z'$, and let $y$ be an essential arc of $\gamma \cap Z'$. Then
\[
d_Z(y, \pi_Z(\gamma)) \le 2\log(2k) +2.
\]
\end{lemma}



\begin{proof}
To prove the lemma, we consider $y$ as an arc of $X_Z$ as follows: for each endpoint $p \in y \cap \partial Z'$ append to $y$ a proper ray in $X_Z$ starting at $p$ which meets $Z'$ only at $p$. Let $y^*$ denote the resulting essential arc of $X_Z$. (Under the implicit/canonical identification $\A(\int(Z')) \cong \A(X_Z)$, $y$ and $y^*$ are identified.) Now we claim that as isotopy classes of arcs in $X_Z$, $y^*$ and $\gamma$ have at most $2k$ essential intersections, thus proving the lemma.

\begin{figure}[htbp]
\begin{center}
\includegraphics[width= 1 \textwidth]{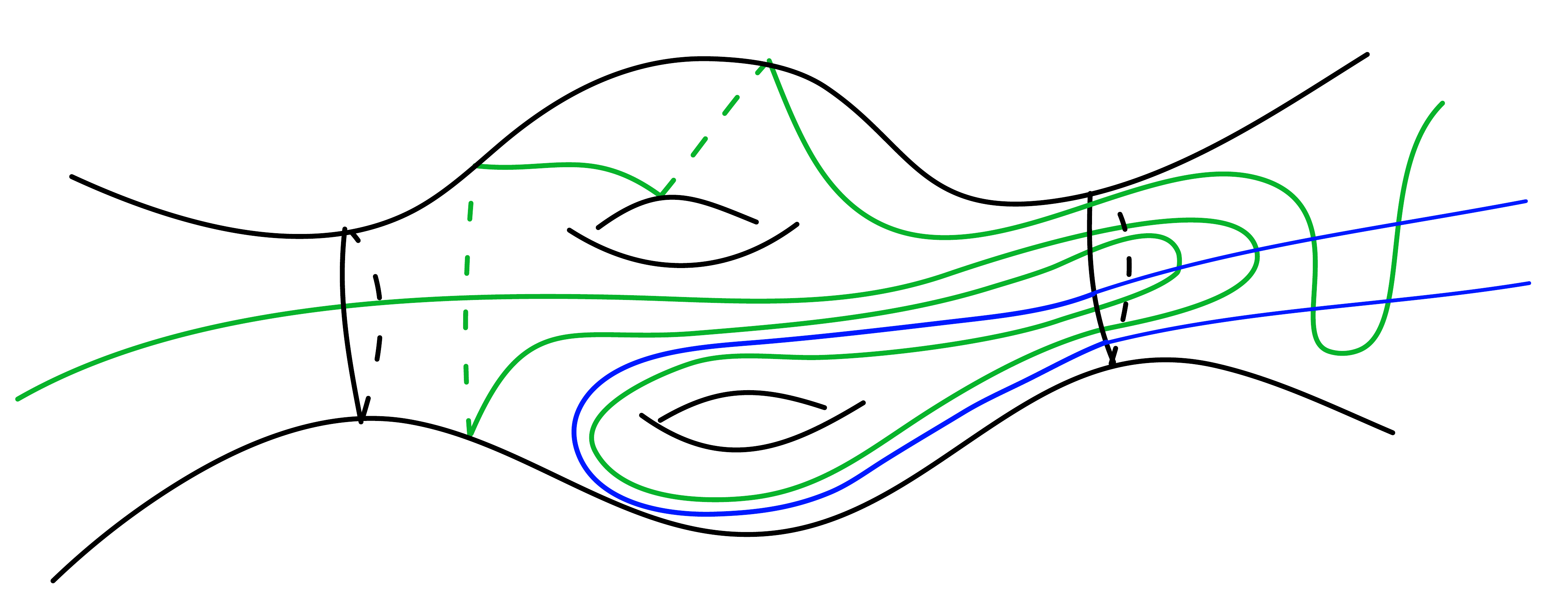}
\caption{The arc $\gamma$ (green) is in $2$-position with $\partial Z$. The arc $y^*$ (blue) intersects $\gamma$ essentially no more than twice.}
\label{fig:2-position}
\end{center}
\end{figure}

For this, first push $y^*$ slightly to one side of itself, so that $y^*$ and $\gamma$ are transverse. 
Then each point $p_i \in y^* \cap \partial Z'$ ($i=1,2$) is contained in no more than $k$ nested subsegments of $\partial Z'$ which cobound bigons $\mathcal{B}_i$ with $\gamma$, as in the definition of $k$-position. See \Cref{fig:2-position}.

Since each component of $X_Z \ssm \int (Z')$ is a disk or annulus, it follows that there is an isotopy supported in $X_Z \ssm \int (Z' \cup \mathcal{B}_1 \cup \mathcal{B}_2)$ which removes all of the intersections between $\gamma$ and $y^*$ which are not contained in $\mathcal{B}_1 \cup \mathcal{B}_2$. Hence, $\gamma$ and $y^*$ have at most $2k$ essential intersections. This completes the proof.
\end{proof}

Because $\int_\tau(Y)$ is an \emph{open} subsurface representative of $Y$, it does not provide a good representative of $\partial Y$. So for the proof of \Cref{prop:proj_close} we do the following: Let $Y_n$ ($n\ge1$) be the exhaustion of $\int_\tau (Y)$ by subsurface isotopic to $Y$ obtained by removing the open $\epsilon_n$--neighborhood of $\partial_\tau Y$. Here, $\epsilon_n \to 0$ as $n \to \infty$, and distance is taken with respect to the flat metric $q$.
Note that by our definitions, each $\A(Y_n)$ is naturally identified with $\A(Y)$.
When $Y$ is not an annulus, will use the property that, through this identification, for any curve or arc $a$ in $X$, the collection of curves and arcs in $\A(Y)$ given by $a \cap Y_n$ eventually agrees with the collection given by $a \cap \int_\tau(Y)$.

\begin{proof}[Proof of \Cref{prop:proj_close}]
First suppose that $Z$ is not an annulus.

As above, let $Y_n$ be the exhaustion of $\int_\tau(Y)$ introduced above. To keep notation simple, we set $Y = Y_n$ for $n$ sufficiently large (to be determined later). 
Note that by construction $\partial Y$ and $\partial_\tau Y$ are disjoint in $X$. Hence, if $\gamma$ is any essential component of the preimage of $\partial Y$ in $X_Z$, then $\gamma$ is disjoint from the preimage of $\partial_\tau Y$. Letting $y$ be any essential component of $\gamma \cap Z_\tau$, this shows that $d_Z(y, \pi_Z^\tau(\partial_\tau Y)) \le 1$. Hence, it suffices to bound the distance between $y$ and $\pi_Z(\gamma)$ in $\A(Z)$. This will follow from \Cref{lem:no_nest}, once we show that $\gamma$ is in $2$--position with respect to $\partial Z_\tau$ in $X_Z$.

Suppose that this were not the case; that is, suppose that there is a point $p \in \partial Z_\tau$ which is contained in $3$ nested subsegments of $\partial Z_\tau$, each of which cobounds a bigon with a subarc of $\gamma$ contained in $X_Z \ssm \int(Z_\tau)$. We now lift this picture to $\til X$ to produces a component $C^\tau_Z$ of the preimage of $Z_\tau$ in $\til X$, a point $\til p \in \partial C^\tau_Z$, and arcs $\til \gamma_1, \til \gamma_2, \til \gamma_3$ in the preimage of $\gamma$ 
 that cobound bigons with nested subsegments of $\partial C^\tau_Z$ containing $\til p$. (The arcs are ordered so that $\til \gamma_1$ is innermost, i.e, closest to $\til p$, and $\til \gamma_3$ is outermost.) Let $B$ be the bigon cobounded by a subarc of $\partial C^\tau_Z$ and $\til \gamma_3$.

Since these arcs are components of preimages of $\partial Y$, we can use them to produce a component $C_Y$ of the preimage of $Y$ in $\til X$ such that
\begin{enumerate}
\item some component of $\partial C_Y$ contains $\til \gamma_1$ or $\til \gamma_2$ , and 
\item $C_Y \cap B$ contains a component which is a disk whose sides alternate between subarcs of $\partial C_Y$ and subarcs of $\partial C^\tau_Z$.
\end{enumerate}
For simplicity, we assume that $\til \gamma_1$ is contained in the component of $\partial C_Y$ in item $(1)$ and denote the other component of $\partial C_Y \cap B$ that appears in item $(2)$ by $\til \eta$. (It may be that $\til \eta = \til \gamma_2$.)
See \Cref{bounded_bigons}.
\begin{figure}[htbp]
\begin{center}
\includegraphics[width=.5 \textwidth]{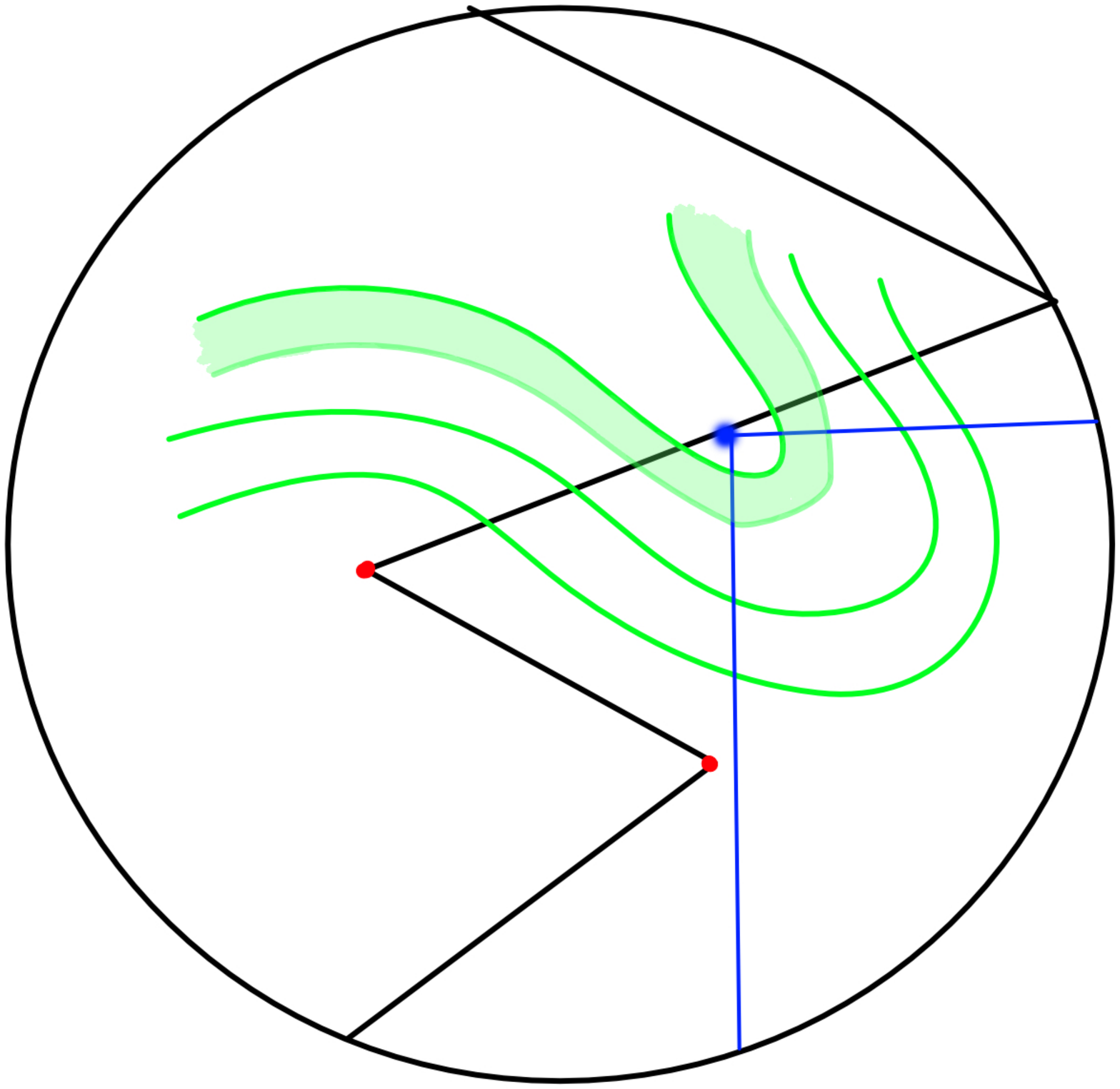}
\caption{The component $C_Y$ (green) and its intersection with $l^\pm$ (blue). The intersection of the black zig-zag arc with $\til X$ give components of $\partial C^\tau_Z$.}
\label{bounded_bigons}
\end{center}
\end{figure}

Now since the leaves of $\lambda^\pm$ are in minimal position with respect to $Z_\tau$ (\Cref{lem:overlap_tau} and the comment that follows), the intersection of each leaf of $\til \lambda^\pm$ (the lifts of $\lambda^\pm$ to $\til X$) with $C^\tau_Z$ is connected.
Hence, we may choose subrays $l^\pm$ in $\til \lambda^\pm$ which are based at $\til p$ and are disjoint from $\int (C^\tau_Z)$. By condition $(2)$ above, each of $l^\pm$ must pass first through $\til \gamma_1$ and then through $\til \eta$. Moreover, each of $l^\pm \cap C_Y$ is an arc joining distinct components of $\partial C_Y$, once $n$ is sufficiently large, again by \Cref{lem:overlap_tau} this time applied to $Y$. 

If $Y$ is also nonannular, then we conclude that the two leaf segments $\lambda^\pm \cap C_Y$ project to give homotopic components of $\lambda^+ \cap Y$ and $\lambda^- \cap Y$. 
Once $n$ is sufficiently large, these gives homotopic components of $\lambda^+ \cap \int_\tau (Y)$ and $\lambda^- \cap \int_\tau(Y)$ and by \Cref{lem:overlap_tau} we see that $d_Y(\lambda^+, \lambda^-) \le 2$. This contradiction shows that $\gamma$ is in $2$-position with $\partial Z_\tau$ and completes the proof when neither $Y$ nor $Z$ are annuli.

If $Y$ is an annulus, then as remarked above
$\int_q(Y)$ is a flat annulus and $\int_q(Y) \subset \int_\tau(Y)$. 
The argument above produces rays $l^\pm$ such that $l ^\pm \cap C_Y$ are leaf segments of $\til \lambda^\pm$. We claim that these segments project to disjoint arcs of $Y$. Otherwise, there is a deck transformation $g$ that stabilizes $C_Y$ such that $(g \cdot l^- \cap l^+) \cap C_Y$ is nonempty. In this case, we have $g\cdot B \cap B \neq \emptyset$ (where $B$ is the bigon from above) and so $g \cdot C^\tau_Z \cap C^\tau_Z \neq \emptyset$. 
Since, $Z$ is $\tau$-compatible, this implies that $g$ stabilizes $C^\tau_Z$ and contradicts our assumption that $Y$ and $Z$ overlap.

Hence, we conclude that $l ^\pm \cap C_Y$ project to disjoint leaf segments in $Y \subset \int_\tau(Y) \subset X_Y$. 
For $n$ large as above, these leaves do not intersect again before exiting $\int_\tau(Y)$ in $X_Y$.
Since leaves of $\lambda^\pm$ intersect at most once outside the maximal open flat annulus $\int(Y_q) \subset X_Y$, this produces representatives of $\pi_Y(\lambda^-)$ and $\pi_Y(\lambda^+)$ intersecting at most twice.  Hence, $d_Y(\lambda^+, \lambda^-) \le 5$ giving the same contradiction as before.


It remains to establish the proposition when $Z$ is an annulus. Since $\pi_Z(\partial Y) \neq \emptyset$, there is a $\tau$-edge in the lift of $\partial_\tau Y$ to $X_Z$ that joins boundary components of $Z_\tau$ on opposite sides of the core curve of $X_Z$. Fix any such $\tau$-edge $a$. Since $\int_q(Z) \subset \int_\tau(Z)$, $a$ crosses $\int_q(Z)\subset X_Z$ and has its endpoints in $X_Z \ssm \int_q(Z)$. Recall the definition $a^* \in \pi_Z^\tau(\partial_\tau Y)$ from the discussion following \Cref{pi Y tau def}.


Let $\gamma$ be any essential component of the lift of $\partial Y$ to $X_Z$ and let $\gamma_q$ be its geodesic representative in the $q$-metric. As before, the preimage of $\partial Y$ in $X_Z$ can be made disjoint from $\gamma_q$ since $Y$ is $q$-compatible. Since $\gamma_q$ is a $q$-geodesic, its intersection with $\int_q(Z) \subset X_Z$ is contained in a single saddle connection $b$.  Moreover, because $Y$ is $\tau$-compatible, 
any saddle connection of $\partial_q Y \subset X$
intersects any $\tau$-edge of $\partial_\tau Y \subset X$ at most once in its interior (see e.g. \cite[Theorem 5.3]{veering1}) and so $a$ and $b$ intersect at most once. If $a^*$ is any extension of $a$ to a complete $q$-geodesic in $X_Y$, then we have by the same
Gauss--Bonnet argument as for the proof of \Cref{lem:overlap_tau}
that $\gamma_q$ intersects $a^*$ at most three times.
We conclude that
\[
\diam_{Z} \big(\pi_Z^\tau(\partial_\tau Y) \cup \pi_Z(\partial Y) \big) \le 2+ d_Z(a^*,\gamma_q)\le 6,
\]
and the proof is complete.
\end{proof}

%% file: immersions.tex

\section{From immersions to covers}
\label{sec:immersions}


When defining the subsurface projection $\pi_Y$ for a subsurface $Y \subset S$ we consider preimages of curves in the cover $S_Y$ associated to $Y$. Of course, the same operation can be done for any cover of $S$ corresponding to a finitely generated subgroup of $\pi_1(S)$. 
The main theorem of this section (\Cref{thm:immersion bound}, which is \Cref{th:intro_3}
in the introduction) gives a concrete explanation for why 
these more general projections do not capture additional information.
This will be an essential ingredient for the proof of \Cref{th:closed_sub_dichotomy}.

First, for a finitely generated subgroup $\Gamma < \pi_1(S)$, let $S_\Gamma
\to S$ be the corresponding cover.  If $Y\subset S_\Gamma$ is a
compact core the covering map restricted to $Y$ is an immersion and we
say that $Y\to S$ {\em corresponds to $\Gamma<\pi_1(S)$.}  
Lifting to
the cover induces a (partially defined) map of curve and arc graphs which we
denote $\pi_Y \colon\A(S)\to \A(Y)$, and we define $d_Y(\alpha,\beta)$ and
$\diam_Y$ accordingly. 
Note that these constructions depend on
$\Gamma$ and not on the choice of $Y$, and that $\pi_Y$ agrees with the usual
subsurface projection when $Y \subset S$.

The goal of this section is the following theorem, which may be of
independent interest:

\begin{theorem}[Immersion to cover] \label{thm:immersion bound}
There is a uniform constant $M\le 38 $ satisfying the following: 
Let $S$ be a surface and let $Y \to S$ be an immersion corresponding
to a finitely-generated $\Gamma < \pi_1(S)$. Then either
$\mathrm{diam}_Y(\A(S)) \le M$,
or $Y \to S$ is homotopic to a
finite cover $Y \to W$ for $W$ a subsurface of $S$. 
\end{theorem}


\Cref{thm:immersion bound} will follow as a corollary of the
following statement:

\begin{theorem} \label{thm:immersion bound lambda}
Let $Y \to S$ be
an immersion corresponding to $\Gamma < \pi_1(S)$, and let
$\lambda^\pm$ be a transverse pair of foliations without saddle connections.
If $d_Y(\lambda^-,\lambda^+)\ge 37$, then $Y \to S$ is homotopic to a
finite cover $Y \to W$ for $W$ a subsurface of $S$. 
\end{theorem}

Let $q$ be a quadratic differential whose horizontal and vertical
foliations are $\lambda^\pm$, and let $X$ denote $S$ endowed with
$q$. 
For curve or arc $\delta$, denote its horizontal length
with respect to $q$ by $h_q(\delta)$ and its vertical length with
respect to $q$ by $v_q(\delta)$. For a homotopy class we let $h_q$ and
$v_q$ denote the minima over all representatives.
Recall that a multicurve $\gamma$ is \emph{balanced} at $q$ if
$h_q(\gamma) = v_q(\gamma)$. For the quadratic differential $q$ and a
multicurve $\gamma$, there is always some time $t \in \mathbb{R}$ such that
$\gamma$ is balanced at $q_t$, where $q_t$ is the image of $q$ under
the Teichm\"uller flow for time $t$.

We let $X_Y=X_\Gamma$ denote the associated cover and
recall from \Cref{sec:compatibility}, the definition of
the $q$-hull $Y_q \subset X_Y$.
If $d_Y(\lambda^-,\lambda^+) \ge 5$ then $Y_q$ is embedded in $X_Y$,
by which we mean the map $\hat\iota_q$ is an embedding on $Y \ssm \partial_0 Y$ (\Cref{prop: q_compatible}).
In the language of the previous section $Y_q$ is a core of $X_Y$ and $\partial Y_q \subset X_Y$ is a collection of locally geodesic curves and proper arcs (of finite $q$-length).

If $\alpha,\beta$ are curves or properly embedded arcs in $X$, then we let
$$ j_Y(\alpha,\beta)$$
denote the minimum, over all components $a$ of $\pi_Y(\alpha)$ and $b$
of $\pi_Y(\beta)$, of the number of intersection points of $a$ and
$b$. We may also use the same notation if $\alpha,\beta$ are laminations in $S$,
or essential curves or properly embedded arcs in $Y$.
The following inequality comes essentially from \cite{rafi2005characterization}:

\begin{lemma} \label{lem:length}
Suppose $Y_q$ is embedded in the cover $X_Y$ and $\partial Y$ is balanced with respect to
$q$. Then for any essential curve or arc $\delta$ in $Y_q$, we have
\[
4 \cdot \ell_q(\delta) \ge j_Y(\lambda, \delta) \cdot \ell_q(\partial Y),
\]
for $\lambda$ equal to either $\lambda^+$ or  $\lambda^-$.
\end{lemma}

\begin{proof}
We show the inequality for $\lambda^+$. First recall that $Y_q$ 
contains a union of 
maximal vertical strips $S_1, \ldots, S_n$ with disjoint, singularity--free interiors 
having the property that $h_q(\partial Y) = 2 \sum_i h_q(S_i)$. Here, $h_q(S_i)$ denotes the width of the strip $S_i$. For details of this construction, see \cite[Section 5]{rafi2005characterization}. If $\delta$ is an essential curve or arc of $Y_q$, then $\delta$ crosses each strip $S_i$ at least $j_Y(\delta,\lambda^+)$ times since each strip is foliated by segments of $\lambda^+$. Hence 
\[
h_q(\delta) \ge \sum_i h_q(S_i) \cdot j_Y(\delta,\lambda^+) = \frac{1}{2} h_q(\partial Y) \cdot j_Y(\delta,\lambda^+).
\]
Since $\partial Y$ is balanced, $h_q(\partial Y) = v_q(\partial Y)$
and so $\ell_q(\partial Y) \le 2h_q(\partial Y)$.
 We conclude 
\[
\ell_q(\delta) \ge h_q(\delta) \ge \frac{1}{4} \ell_q(\partial Y) \cdot j_Y(\delta,\lambda^+)
\]
as required.
\end{proof}

We now proceed with the proof of  \Cref{thm:immersion bound lambda}.

\begin{proof}
We may suppose that $d_Y(\lambda^-,\lambda^+)\ge 5$
so that $Y_q$ is embedded in $X_Y$. If $Y$ is an annulus, then
$\int(Y_q) \subset X_Y$ is a flat annulus which must cover a flat 
annulus in $X$. So we now suppose that $Y$ is not an annulus.
We may further assume, applying the Teichm\"uller flow to $q$ if
necessary, that $\partial Y_q$ is balanced.

Let $P$ be either $S^1$ or $\RR$.
If $\gamma \colon P \to X_Y$ is a parameterization of a
 boundary component of $Y_q$ (see \Cref{Y_q_cartoon_2}), then
we say that a \emph{re-elevation} of $\gamma$ is any lift to $X_Y$  of
$\mathbb{R} \to P \overset{\gamma}{\to} X_Y \overset{p}{\to} X$,
where $\mathbb{R} \to P$ is the universal cover, and
we say that the re-elevation is \emph{essential} if it meets $Y_q$
essentially. 
From nonpositive curvature of the metric, a re-elevation is essential if and only if it meets $\int(Y_q)$. 

We note that with this terminology, $Y \to X$ covers a subsurface of
$X$ if and only if there are no essential re-elevations of components
of $\partial Y_q$, since in this case $p^{-1}(p(\partial Y_q)) \cap Y_q
= \partial Y_q$ (see \cite[Lemma 6.6]{veering1}).  Thus our goal now
is to prove that if $d_Y(\lambda^-,\lambda^+)$
is sufficiently large, independent of $Y$ and $X$, then there are no
essential  re-elevations of $\partial Y_q$.

Let $g \colon \mathbb{R}\to X_Y$ be a re-elevation of $\gamma$, let $(a,b)\subset
\R$ be a component of $g^{-1}(\int_q Y)$, and let 
$\widetilde{\gamma}$ denote the restriction of $g$ to $[a,b]$. This
is a geodesic path (possibly with self-intersection) with
endpoints in $\partial Y_q$, and if $g$
is essential then $\widetilde\gamma$ may be chosen to be essential.


We now look for a restriction of $\widetilde\gamma$ to an essential simple
arc or curve. Let $d \in [a,b]$ be the supremum over $t \in [a,b]$ for
which $\widetilde{\gamma}|_{[a,t]}$ is an embedding.

\subsubsection*{\bf Case 1:} $d<b$. Then there exists
$c\in[a,d)$ such that $\widetilde{\gamma}(d) = \widetilde{\gamma}(c)$.
  Thus $x=\widetilde{\gamma}|_{[c,d]}$ is an embedded loop.

\subsubsection*{\bf Case 1a:} $x$ is an essential loop, which we name $\sigma$. 

In the case where $P = S^1$, divide $\R$ into fundamental domains for 
the covering map $\R\to
S^1$, so that $c$ is a boundary point of a fundamental domain,
and let $l\ge 0$ be the number of full fundamental domains
contained in $[c,d]$. When $P = \RR$, set $l=0$.
Hence, as $\gamma$ is a component of $\partial_q Y$,
\[
\ell_q(\sigma) \le \ell_q(\widetilde{\gamma}|_{[c,d]}) \le (l+1) \cdot \ell_q(\partial_q Y).
\]
By \Cref{lem:length},
\[
 j_Y(\lambda^\pm, \sigma) \cdot \ell_q(\partial Y) \le 4 \cdot \ell_q(\sigma) \le 4 (l+1) \cdot \ell_q(\partial Y)
\]
and so $ j_Y(\lambda^\pm, \sigma) \le 4 (l+1)$.

The Gauss--Bonnet theorem for the Euclidean cone metric on $Y_q$ implies that the number of singularities in the interior of $Y_q$ is no more than $2|\chi(Y)|$ and since $\widetilde{\gamma}|_{[c,d]}$ is an embedded loop, it visits each of these singularities of $Y_q$ at most once. As $\widetilde{\gamma}|_{[c,d]}$ contains at least $l$ singularities interior to $Y_q$, we obtain $l \le 2|\chi(Y)|$. Combining this with the inequality arrived at above, we conclude that 
\begin{equation}\label{bound jY}
j_Y(\lambda^\pm, \sigma) \le 4(2\cdot |\chi(Y)|+1).
\end{equation}
We now invoke \Cref{lem:effective_B} (and in particular \cref{eqn:i bound 8})
to conclude that 
$d_Y(\lambda^\pm, \sigma) \le 15$. Hence, $d_Y(\lambda^-,\lambda^+) \le
30$ as required.


\subsubsection*{\bf Case 1b:}
If $x=\widetilde\gamma|_{[c,d]}$ is inessential, it still cannot be
null-homotopic since $g$ is a geodesic path, so it must be
peripheral. That is, either $x$ bounds a punctured disk in $Y_q$ or
$x$ cobounds an annulus $A$  with a boundary component $u$ of $Y_q$. We
claim that in the latter case the endpoint
$\widetilde{\gamma}(a)$ does not lie on $u$.

\realfig{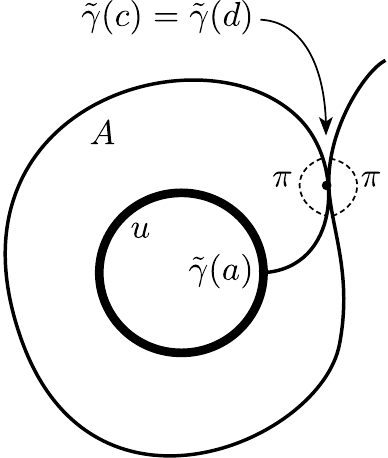}{When $\widetilde \gamma$ meets itself but glances off the annulus $A$, we obtain a geodesic loop homotopic to the boundary component $u$.}

Suppose otherwise. Then there are two possibilities. If $\widetilde\gamma|_{[d,d+\epsilon)}$ does not enter $A$, then $x$ is a
  $q$-geodesic loop: at $\widetilde\gamma(c)$ it subtends an angle of at
  least $\pi$ inside $A$, and at $\widetilde\gamma(d)$ it subtends an
  angle of at least $\pi$ outside $A$. (See \Cref{annulus-tangent}). But this is 
  a contradiction -- $x$ cannot be a geodesic representative of $u$
  and not equal to $u$.

\realfig{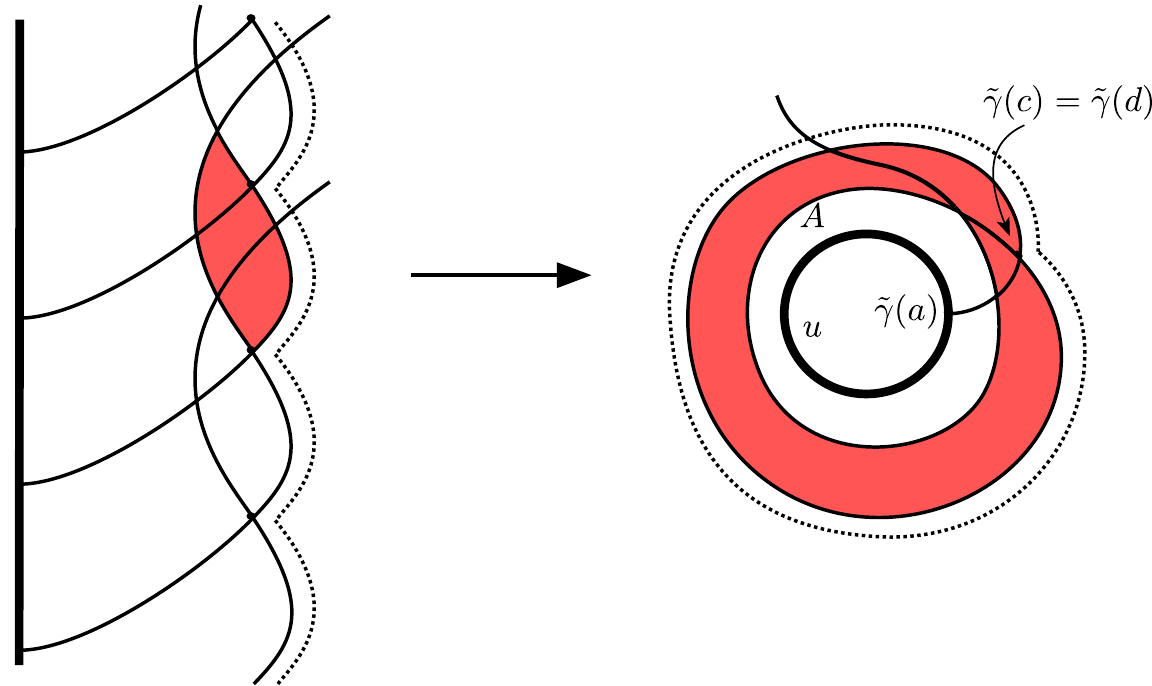}{On the right is the annulus $A$ formed when $\gamma$ crosses itself at its first intersection point. Its thickening $A'$ is indicated by the dotted circle. At left is the universal cover of $A'$ where one can see the bigon between translates of the lift of $\widetilde\gamma$.}

If $\widetilde\gamma|_{[d,d+\ep)}$ enters $A$ for small $\ep$, then we claim
there is an immersed $q$-geodesic bigon cobounded by arcs of 
$\widetilde{\gamma}$, which contradicts nonpositive curvature.
Indeed, thicken $A$ slightly to an annulus $A'$ and 
let $t>d$ be
the smallest value for which $\widetilde\gamma(t)$ meets $\partial A' \ssm u$. Consider
the lifts of $\widetilde\gamma|_{[a,t]}$ to the universal cover $\widetilde A'$
of $A'$. This cover is an infinite strip, and the lifts form a
$\Z$-periodic sequence of arcs connecting the two boundary
components. One may order the lifts by their endpoints on (either)
boundary component, and then we see that, since each lift crosses the
lift above it at a preimage of $\widetilde\gamma(d)$, two consecutive
lifts must intersect at least twice. This produces an immersed bigon downstairs, and our
contradiction (see \Cref{annulus-bigon}).


\realfig{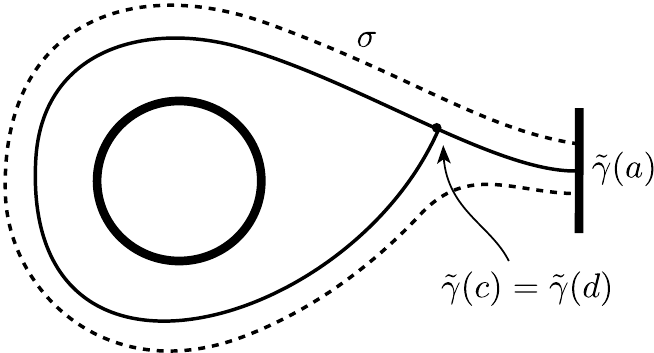}{$\sigma$ is formed from the self-intersection of $\widetilde\gamma$. The inner circle represents $u$ or the puncture.}
We conclude that the concatenation $\sigma$ of $ \widetilde{\gamma}|_{[a,d]}$ followed by $\widetilde{\gamma}|_{[a,c]}$  
traversed in the opposite direction
is an essential arc of $Y_q$ which is homotopic to an embedded arc
(\Cref{sigma-arc}).
Moreover, since $x$ is embedded, $\sigma$ meets no singularity
interior to $Y_q$ more than twice. Hence, if we let $l\ge 0$ be the
number of fundamental domains in $[a,d]$, defined exactly as above,
then $\ell_q(\sigma) \le 2 \ell_q(\widetilde{\gamma}|_{[a,d]}) \le 2(l+1)
\cdot \ell_q(\partial Y)$. Using the same reasoning as in the previous
case, we conclude that $ j_Y(\lambda^\pm, \sigma) \le 8 (l+1)$. Using
the fact that $\sigma$ meets no singularity of $Y_q$ more than twice, we have
$l \le 4 |\chi(Y)|$ and hence 
\[
j_Y(\lambda^\pm, \sigma) \le 8 (4\cdot |\chi(Y)|+1).
\]
This time we apply \cref{eqn:i bound 32} to conclude that $d_Y(\lambda^\pm \sigma) \le 18$ and so $d_Y(\lambda^-,\lambda^+) \le 36$. 

\subsubsection*{\bf Case 2:}  $d=b$. This final case is handled just like Case 1a,
except that $\sigma = \widetilde\gamma$ is an essential arc with embedded
interior, rather than an essential loop. 
\end{proof}

\Cref{thm:immersion bound} now follows as a corollary: 

\begin{proof}
Suppose that $Y$ does not cover a subsurface of $S$.
Let $\alpha, \beta$ be two curves with nontrivial projection to $Y$ and let $(\lambda^\alpha_n)_{n\ge0}$ and $(\lambda^\beta_n)_{n\ge0}$ be any sequences of filling laminations such that $\alpha$ is contained the Hausdorff limit of $\lambda^\alpha_n$ and $\beta$ is contained in the Hausdorff limit of $\lambda^\beta_n$. For example, if $f \colon S \to S$ is any pseudo-Anosov homeomorphism then we can take the stable and unstable laminations of $T^n_\alpha \circ f \circ T^n_\beta$ which is pseudo-Anosov for all but finitely many $n \in \Z$. Here $T_\gamma$ denotes the Dehn twist about the curve $\gamma$.

Now let $q_n = q(\lambda^\alpha_n,\lambda^\beta_n)$ be the holomorphic
quadratic differential whose vertical and horizontal foliations are
determined by $\lambda^\alpha_n$ and $\lambda^\beta_n$,
respectively. By \Cref{thm:immersion bound lambda}, then $d_Y(\lambda^\alpha_n,\lambda^\beta_n)\le 36$ for each $n\in \Z$. However, for large enough $n\ge0$, $\pi_Y(\alpha) \subset \pi_Y(\lambda^\alpha_n)$ and $\pi_Y(\beta) \subset \pi_Y(\lambda_n^\beta)$ and we conclude that $d_Y(\alpha,\beta) \le 36$. Since $\alpha$ and $\beta$ were arbitrary curves with nontrivial projection to $Y$, we conclude that $\mathrm{diam}_Y(\A(S)) \le  38$.
\end{proof}

%% file: bounds.tex

\section{Uniform bounds in the veering triangulation}
\label{sec:bounds}

In this section we produce two estimates relating sections of the veering triangulation to
subsurface projections in the \emph{original} surface $X$. In \Cref{prop:closed_distance},
we show that for a $\tau$-compatible subsurface $Y$ of $X$, the $\pi_Y^\tau$-projections
of the top and bottom sections of $T(\partial_\tau Y)$ are near the projections of
$\lambda^+$ and $\lambda^-$, respectively,  to $\A(Y)$. In \Cref{lem:slowed_progress}, we
show that if sections $T_1$ and $T_2$ have sufficiently far apart $\pi_Y^\tau$-projections
to $\A(Y)$, then they must differ by at least $|\chi(Y)|$ tetrahedron moves in the veering
triangulation. 
Both estimates will be important ingredients in the proof of \Cref{th:bounding_projections_closed}.

\subsection{Top/bottom of the pocket and distance to $\lambda^\pm$}


Let $Y \subset X$ be a $\tau$-compatible subsurface as in \Cref{sec:compatibility}.
Since $\partial_\tau Y$ is a collection of edges of $\tau$ with disjoint interiors in $X$, 
\Cref{lem:extension} gives that $T(\partial_\tau Y) \neq \emptyset$.
Let $T^\pm \in T(\partial_\tau Y)$ denote the top and bottom sections containing $\partial_\tau Y$.


The following proposition is analogous to \cite[Proposition 6.2]{veering1} in the fully punctured setting. However, more work is needed here to relate the projection of $T^+$ to the projection of $\lambda^+$ in the curve and arc graph of $Y$.

\begin{proposition}[Compatibility with $\tau$ and distance to $\lambda^\pm$] \label{prop:closed_distance}
Let $T^\pm$ be the top and bottom sections in $ T(\partial_\tau Y)$. For any section $Q \ge T^+$, $d_Y(Q, \lambda^+) \le D+1$ and for any $Q \le T^-$, $d_Y(Q,\lambda^-) \le D+1$.
\end{proposition}

Recall that $d_{Y}(Q,\lambda^-)$ means $d_Y(\pi_Y^\tau(Q), \pi_Y(\lambda^-))$. 

\begin{proof}
We begin by remarking that if $\int_\tau(Y)$ contains no singularities of $q$ other than punctures (i.e. if $\PP_Y = \sing(q) \cap \int_\tau(Y)$ in $\overline X$), then the argument from \cite[Proposition 6.2]{veering1} carries through and gives a better constant. This includes the situation where $Y$ is an annulus.

In the general (nonannular) case, we show that there is an embedded edge path $p$ in $\Pi_*(Q) = \Delta_{Q}$ which projects to an essential arc of $\int_\tau (Y)$ and is isotopic to a properly embedded arc of $\int_\tau (Y) \cap \lambda^+$. 
Together with \Cref{lem:overlap_tau}, this shows that
 \[d_{Y}(Q,\lambda^+) = d_{Y}(\pi_Y^\tau(Q),\lambda^+) = \diam_{Y}(\pi_Y^\tau (Q)) + \diam_{ Y}(\lambda^+) \le D+1\]
as required (the proof for $\lambda^-$ is identical). 


\realfig{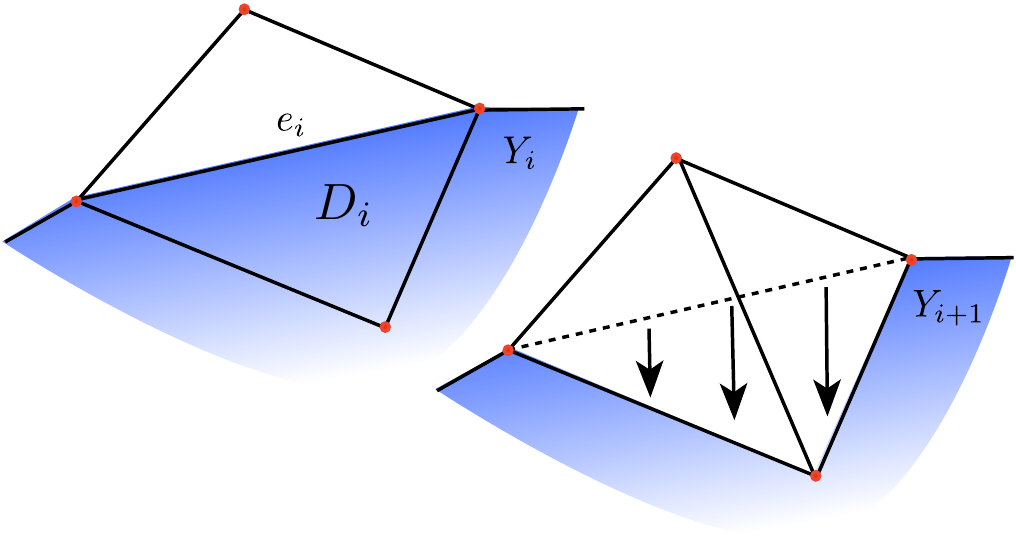}{The first step: The upward diagonal exchange along the boundary and removing the triangle $D_i$. Vertices lie in $\sing(q) \cup \PP$.}

Let $T_0, T_1, T_2 , \ldots$ be any sequence of sections through upward tetrahedron
moves starting with $T_0 = T^+$  such that $Q = T_j$ for some $j\ge 0$
(\Cref{lem:connectivity}), and let $\Delta_i = \Pi_*(T_i)$ be the corresponding
$\tau$-triangulations of $X$. Note that $\Delta_{i+1}$ is obtained
from $\Delta_i$ by a single diagonal exchange.

Since $T_0 \in T(\partial_\tau Y)$, $\Delta_0$ contains a subcomplex $Y_0$ which triangulates the image of $Y_\tau$ under the covering $X_Y \to X$.
We inductively construct a sequence of subcomplexes $Y_0 \supset Y_1  \supset \ldots \supset Y_{n+1}$ of $\Delta_0 \subset X$ satisfying certain conditions. For this, we say that each $Y_i$ comes with a \emph{boundary} $\partial_\tau Y_i$, which is defined inductively below, and we set $\int_\tau (Y_i) = Y_i \ssm \partial_\tau Y_i$. For $Y_0 $, $\partial_\tau Y_0 = \partial_\tau Y$ so that $\int_\tau (Y) = Y_0 \ssm \partial_\tau Y_0$ is isotopic to $Y$.
 
To construct $Y_{i+1}$ from $Y_{i}$ observe that the upward exchange from $\Delta_{i}$ to $\Delta_{i+1}$ either occurs along an edge not meeting $Y_{i}$, in which case we set $Y_{i+1} = Y_{i}$, or 
it must occur along an edge $e_{i}$ of $\partial_\tau Y_{i}$. Otherwise, the edge $e_{i}$ lies in the interior of $Y_{i} \subset Y_0$ where it is wider (with respect to $q$) than the other edges in its two adjacent triangles. Hence, the same must be true in the triangulation $\Delta_0$, and we see that $\Pi^*(e_{i})$ is upward flippable in $T^+$. This gives a section $T' \in T(\partial_\tau Y)$ with $T^+ < T'$, contradicting the definition of $T^+$.

Let $Q_i$ be the quadrilateral in $\Delta_i$ whose diagonal is $e_i$. Writing $Q_i$ as
two triangles adjacent along $e_i$, at least one of them, call it
$D_i$, is contained in $Y_i$,
as in \Cref{triangle-push}. If $D_i$ is the only triangle in $Y_i$, set 
$Y_{i+1} = Y_i \ssm (\mathrm{int}(D_i) \cup e_i)$ and
$\partial_\tau Y_{i+1} =( \partial_\tau Y_{i} \cup \partial D_{i}) \ssm e_{i}$.

If the other triangle $D'_i$ of $Q_i$ is also in $Y_i$, set
$Y_{i+1} = Y_i \ssm \int(Q_i)$ and $\boundary_\tau Y_{i+1} = \boundary_\tau Y_i \union
\boundary Q_i \ssm e_i$. 

Let $v_i$ be the vertex of $D_i$ opposite $e_i$ (and $v'_i$ the vertex of $D'_i$ opposite
$e_i$ if we are in the second case). If
$v_i$ (or $v'_i$) is contained in $\partial_\tau  Y_i \cup \PP$, 
or $v_i = v'_i$,
then we set $n=i$, so that $Y_{n+1}$
is the last step of the construction.

\realfig{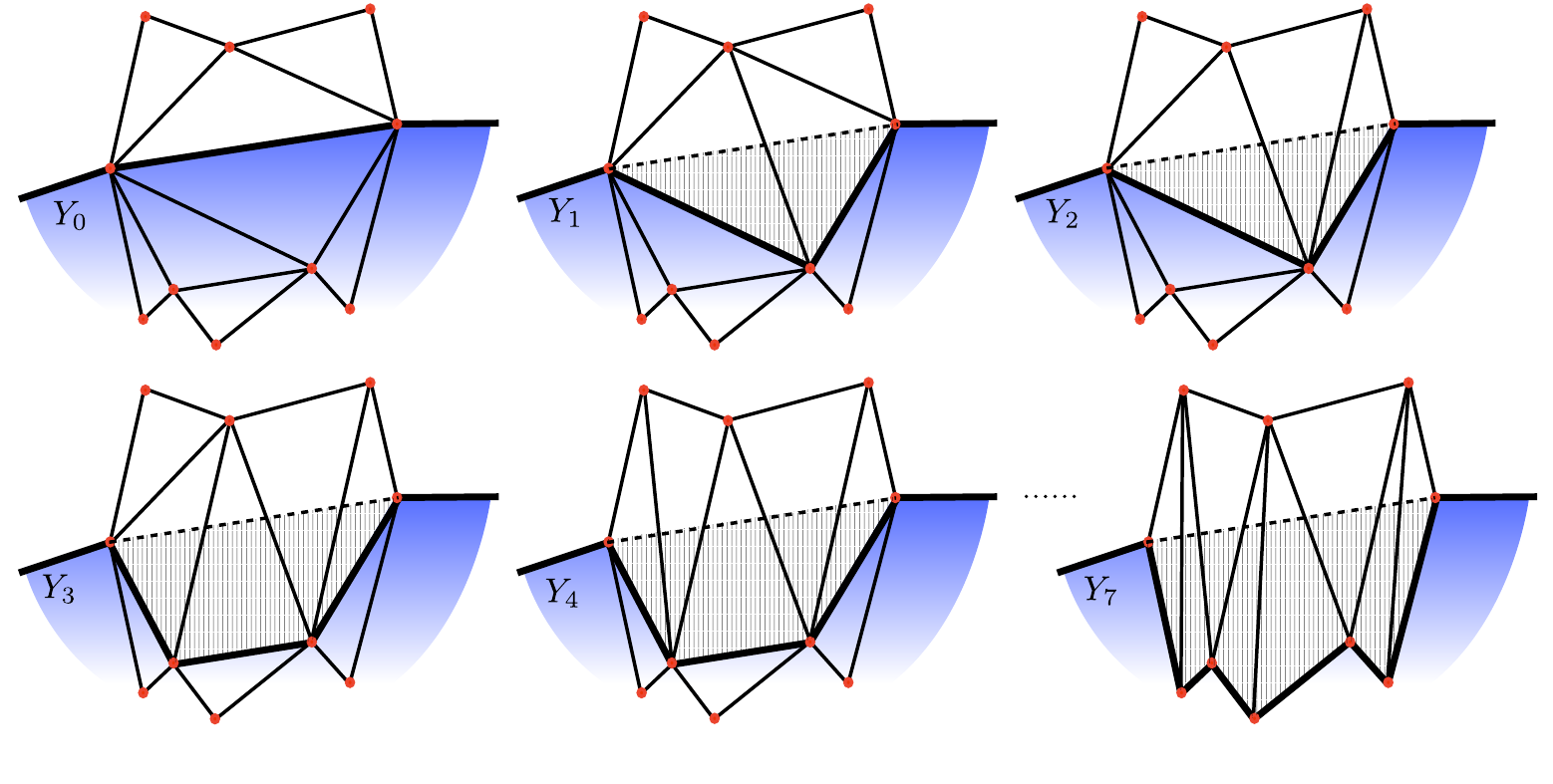}{An example sequence of $Y_i$ (indicated in blue). A component $K$ of $\int_\tau(Y_0)\ssm Y_i$ is
  indicated with its $\lambda^+$ foliation in grey. $\sigma_K$, dotted, is the edge $e_0$.}


The construction has the following properties for $i\le n+1$:
\begin{enumerate}
  \item $Y_i$ is a subcomplex of each of $\Delta_0, \Delta_1, \ldots, \Delta_i$. Moreover, the
    diagonal exchange from $\Delta_{i-1}$ to $\Delta_{i}$ replaces an edge $e_{i-1}$ in either
    $\boundary_\tau Y_{i-1}$ or in the complement of $Y_{i-1}$. 
  \item
    Each component $K$ of $\int_\tau(Y_0)\ssm Y_i$ is a disk foliated by leaves of
    $\lambda^+$. Moreover, $\partial K$ is composed of two arcs $p_K$ and $\sigma_K$,
    where $\sigma_K$ is a $\tau$-edge in $\partial_\tau Y$ and $p_K$ is a path in
    $\partial_\tau Y_i$, and each leaf of the foliation of $K$ meets both $p_K$ and
    $\sigma_K$ at its endpoints.
  \item For each component $K$ of 
$\int_\tau(Y_0)\ssm Y_i$ and each interior vertex $v$ of $p_K$, there is at least one edge $e$ of
    $\Delta_i$ entering $K$ from $v$, and every such edge crosses $\sigma_K$ from top to
    bottom, so that $e > \sigma_K$.
\end{enumerate}

First note that $Y_0$ satisfies the properties for $i=0$ 
and that property $(1)$ holds for all $i\ge 0$ by construction.

Assume property (2) holds for $i\le n$ and let us prove it for $i+1$.
Let $D_i$ be the disk removed
to obtain $Y_{i+1}$, and let $p_{D_i}$ be the other two sides of $\boundary D_i$.

First suppose $e_i$ is in
$\boundary_\tau Y_0$ (which in particular holds for $i=0$). The two edges in $p_{D_i}$ cannot
be in the boundary of any other component $K'$ of $\int_\tau(Y_0) \ssm Y_i$, because
then for some $j< i$ there would have been a $D_j$ whose third vertex $v_j$ was either a
puncture or on $\partial_\tau Y_0$, which implies $j>n$. Thus $D_i$ itself is a component
of $\int_\tau(Y_0)\ssm Y_{i+1}$, and (2) evidently holds, with $\sigma_{D_i} = e_i$. 

Now if $e_i$ is not in $\boundary_\tau Y_0$, it must lie in $p_K$ for some component $K$
of $\int_\tau(Y_0) \ssm Y_i$. 
We have again that $p_{D_i}$ cannot share edges with any
other component of $\int_\tau(Y_0) \ssm Y_i$, and so we obtain a component $K'$ of
$\int_\tau(Y_0)\ssm Y_{i+1}$ by adjoining $D_i \ssm p_{D_i}$ to $K$. The boundary path
$p_{K'}$ is obtained from $p_K$ by replacing $e_i$ by $p_{D_i}$, and the $\lambda^+$
foliation of $K$ extends across $e_i$ to $K'$. The edge $\sigma_{K'}$ is just $\sigma_K$.

\realfig{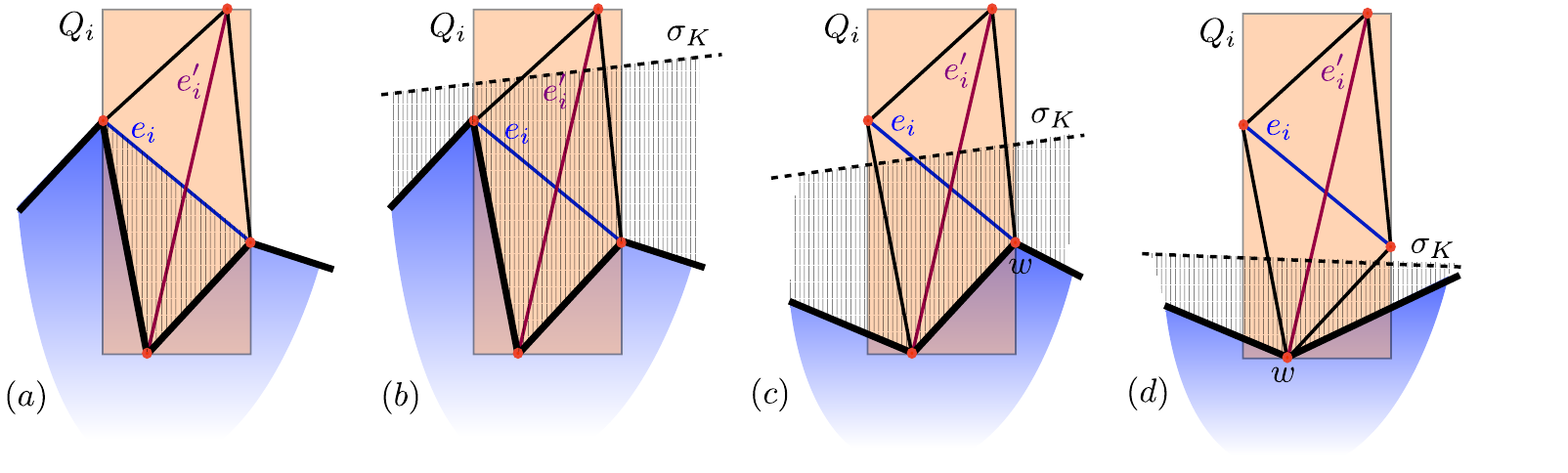}{The proof of property (3)}

Now we consider property (3). Breaking up into cases as in the proof of (2),
consider first the case that $e_i$ is in $\boundary_\tau Y_0$ (\Cref{Prop3}(a)).
Let $Q_i$ be the
quadrilateral defining the move $\Delta_i \to \Delta_{i+1}$. Then the new edge $e'_{i}$ is
the other diagonal of $Q_i$ and must cross $e_i$. This shows that (3) holds for the new
component $K = \int(D_i)$ of $\int_\tau(Y_0) \ssm Y_{i+1}$.

If $e_i$ is in $p_K$ for some component $K$ of $\int_\tau(Y_0) \ssm Y_i$, consider $Q_i$
again (\Cref{Prop3}(b)) and let $D'_i$ be the complementary triangle to $D_i$ in $Q_i$. The two
complementary edges to $e_i$ in $\boundary D'_i$ must, by inductively applying (3), cross
$\sigma_K$. It follows that the new edge $e'_i$ also crosses $\sigma_K$, thus proving (3)
for $i+1$. Here, we are using the general fact that if an edge of a $\tau$-triangle crosses $\sigma$ from top to bottom, then so does the tallest edge of that triangle.

Finally suppose $e_i$ is outside of $Y_i$. If $Q_i$ is not adjacent to $Y_i$ at all then
nothing changes and (3) continues to hold. If $Q_i$ shares a vertex $w$ with $p_K$ for some
component $K$ of $\int_\tau(Y_0) \ssm Y_i$, we have a configuration like 
\Cref{Prop3}(c) or (d). In (c), $w$ is on $e_i$. 
The new edge $e'_i$, which connects to another vertex on $p_K$, must therefore also cross $\sigma_K$ (by transitivity of $<$).
Moreover $w$ is
still adjacent to one of the boundary edges of $Q_i$ which also passes through
$\sigma_K$. Thus (3) is preserved at all vertices of $p_K$. In (d), $w$ is the vertex
meeting $e'_i$ and again we must have $e'_i$ crossing $\sigma_K$. This concludes the proof
of properties (1-3). 


\medskip

\realfig{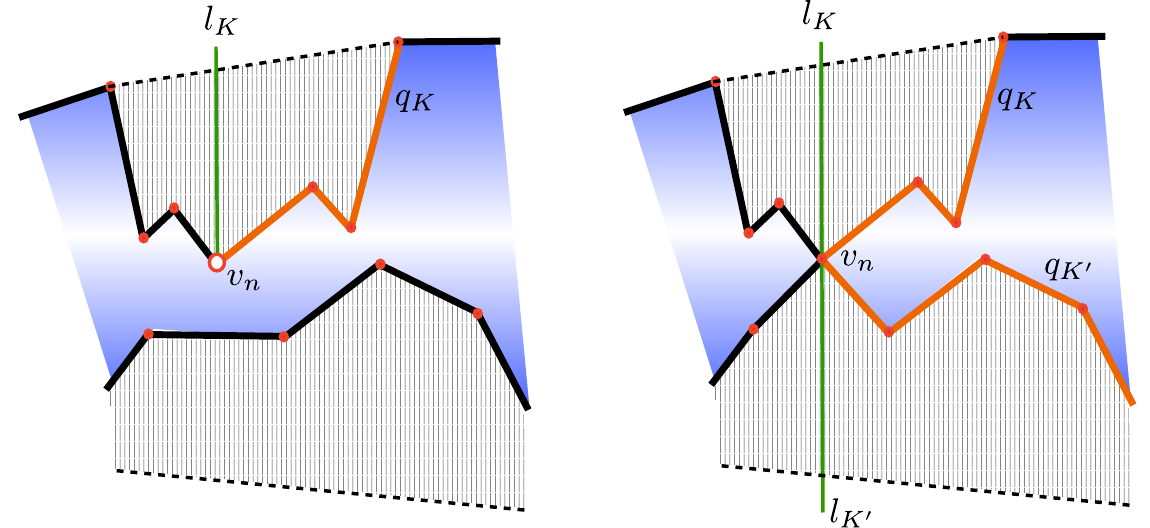}{Two examples of $Y_{n+1}$ and the corresponding $\lambda^+$ arcs and
  $\Delta_i$ arcs. In the first case the point $v_n$ is a puncture so only one leaf $l_K$
and path $q_K$ is needed. In the second case $v_n$ lies on the boundary of two components
of $\int_\tau(Y_0) \ssm Y_{n+1}$.}

The sequence terminates with the diagonal exchange from $\Delta_n$ to $\Delta_{n+1}$
along an edge $e_n$ whose associated triangle $D_n$ has its opposite vertex $v_n$ in
either $\PP$ or the boundary of $Y_n$. 
There are either one or two components $K$ of $\int_\tau(Y_0) \ssm Y_{n+1}$ whose boundary
path $p_K$ contains $v_n$. Examples of these two cases  are shown in \Cref{last-Y}.

For each such component $K$, let $l_K$ be the leaf of $\lambda^+$ adjacent to $v_n$ and
exiting $K$ through $\sigma_K\subset \boundary_\tau Y_0$, as in property (2). Let $q_K$ be
a path along $p_K$ from 
$v_n$ to one of the endpoints of $p_K$. When there is only one component $K$, let
$l=l_K$ and $q=q_K$.  If there are two components $K,K'$, let $l = l_K \union l_{K'}$ and
$q = q_K \union q_{K'}$.

By \Cref{lem:overlap_tau}, $l$ is an embedded essential arc of $\int_\tau Y$, and $q$ and
$l$ define the same element of $\A(Y)$.

Recall that $\Pi_*(Q) = \Delta_j$. 
If  $j\le n+1$ we know that $Y_{n+1}$ is a subcomplex of $\Delta_j$ by (1), and hence we
obtain a bound
$$d_Y(\Delta_j, \lambda^+) \le D+1.$$

If $j>n+1$ we must consider what happens after further transitions. By (3), for each
component $K$ of $\int_\tau(Y_0) \ssm Y_{n+1}$ adjacent to $v_n$, there is an edge
$f_{n+1}$ of $\Delta_{n+1}$ adjacent to $v_n$, and passing through $K$ and exiting through
$\sigma_K$. In particular the defining rectangle $R_{n+1}$ of $f_{n+1}$ passes through the
rectangle of $\sigma_K$ from top to bottom. We claim that there is such an $f_i$ and $R_i$
for each $i\ge n+1$.

\realfig{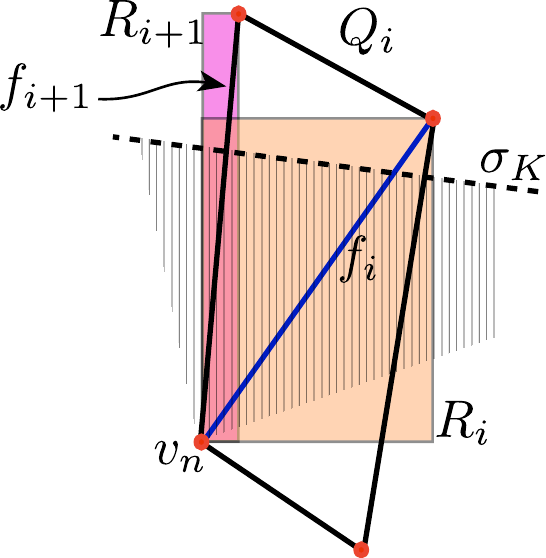}{For $i>n$ we always find an edge in $\Delta_i$ starting at $v_n$ and
  exiting $K$ through $\sigma_K$.}

If this holds for $i$, and the move $\Delta_i \to \Delta_{i+1}$ does not replace $f_i$
then we can let $R_{i+1}=R_i$ and $f_{i+1} = f_i$. If $f_i$ is replaced, there is a
quadrilateral $Q_i$ in $\Delta_i$ of which $f_i$ is a diagonal, and one of the two edges
of $\boundary Q_i$ adjacent to $v_n$ must also cross the $\sigma_K$ rectangle from top to
bottom. (See \Cref{f-cross}). 

Therefore, we can use these edges of $\Delta_j$ (either one or two depending on the number
of components $K$ adjacent to $v_n$) to give an essential path in $\Delta_j$ which gives
the same element of $\A(Y)$ as the leaf path $l$. We conclude again that 
$$d_Y(\Delta_j, \lambda^+) \le D+1$$
which completes the proof.
\qedhere
\end{proof}


\subsection{Sweeping (slowly) through pockets}
The following lemma states that in order to move a definite distance in the curve graph of $Y \subset X$ a certain number of edges, linear in the complexity of $Y$, need to be flipped. It will  needed for the proof of \Cref{th:bounding_projections_closed}.

\begin{lemma}[Complexity slows progress] \label{lem:slowed_progress}
Suppose that $T_1, T_2 \in T(\partial_\tau  Y)$ are connected by no more than $|\chi(Y)|$ diagonal exchanges through sections of $T(\partial_\tau Y)$. Then $d_{Y}(T_1,T_2) \le 2D$.
\end{lemma}

\begin{proof}
We begin with the following claim:

\begin{claim2}
Let $\overline Y$ be a compact surface containing a finite (possibly empty) set $\PP_Y \subset \int(\overline Y)$, let $Y = \overline Y \ssm \PP_Y$, and assume that $\chi(Y)\le -1$. Let $\overline \Delta$ be a triangulation of $\overline Y$ whose vertex set contains $\PP_Y$ and consider the properly embedded graph $\Delta = \overline \Delta \ssm \PP_Y \subset Y$.
If $E$ is a collection of nonboundary edges of $\Delta$ with $|E| \le |\chi(Y)|$,  there is $c \in \A(Y)$ that is nearly simple in $\Delta$ and traverses no edge of $E$.
\end{claim2}

Recall from \Cref{sec:prop_graph} that $c$ is nearly simple in $\Delta$ if it is properly homotopic to path or curve in $\Delta$ that visits no vertex more than twice. Note that the conclusion of the claim is equivalent to the statement that the proper graph $\Delta^{(1)} \ssm \int(E)$ is essential in $Y$.

Applying the claim to $\int_\tau (Y)$ gives $c \in \A(Y)$ which is nearly simple in the graph $\mathrm{cl}_X(\Pi_*((T_1 \cap T_2)^{(1)}))$.
Hence, $d_Y(T_1,T_2) \le 2D$.

Now it suffices to prove the claim:
\begin{proof}[Proof of claim]
First, we blow up the punctures $\PP_Y$ to boundary components. That is, each $v \in \PP_Y$ is replaced with a subdivided circle $S_v$ containing a vertex for each edge adjacent to $v$. Continue to call the resulting surface $Y$ and note that $\Delta$ induces a natural cell structure on $Y$, which we continue to denote by $\Delta$, made up of $n$-gons with $n\le 6$. Obviously, $\chi(Y)$ is unchanged.

Considering $E$ as a possibly disconnected subgraph of $\Delta$, let $E'$
be the edges of $E$ remaining after removing the components of $E$
which are contractible in $Y$ and do not meet $\partial Y$. Hence,
every component of $Y \ssm E'$ is $\pi_1$--injective. We first
claim that some component of $Y \ssm E'$ is neither a disk nor a
boundary parallel annulus. 
Letting $G = E' \union \boundary Y$, we have
$\chi(Y) = \chi(Y\ssm G) + \chi(G)$ since $Y\ssm G$ is
adjoined to $G$ along circles.  Now $\chi(G) = \chi(\boundary Y) +v
-e = v-e$, where $v$ is the number of vertices in $E'\ssm
\boundary Y$ and $e = |E'|$.
If $Y\ssm E'$ consists of only
$d\ge 0$ disks and $a\ge 0$ annuli then we have
$\chi(Y\ssm G) = \chi(Y\ssm E') = d$, and so
$\chi(Y) = d+v-e$. 
Since an annulus must have a vertex of $E'\ssm
\boundary Y$ in its boundary, we 
note that if $d=0$ then $v>0$, and hence $d+v \ge 1$. We
thus have $\chi(Y) \ge 1-|E'|$, and so 
 $|\chi(Y)| \le |E'| -1 \le |E| -1 $, a contradiction.

Now let $U'$ be some component of $Y \ssm E'$ which is neither a
disc nor a boundary parallel annulus. Hence, there is an essential simple
closed curve $\gamma'$ of $Y$ contained in $U'$. As $U = U'
\ssm E$ has a component corresponding to $U'$ minus a collection
of disks, $\gamma'$ is homotopic to a simple curve $\gamma$ in $U$.
Let $c$ be a cell of $\Delta$ which $\gamma$
crosses. Since $c$ is a blown-up triangle, the edges that $\gamma$
crosses, which are in the complement of $E$, are connected along
either vertices of $c$ or arcs of $\boundary Y$ (which are also not in
$E$).
Thus $\gamma$ can be deformed to a curve in $\Delta$ that does not
traverse the edges of $E$. This shows that the proper graph 
$\Delta^{(1)} \ssm \int(E)$ is essential in $Y$ and the claim follows.
\end{proof}
This completes the proof of \Cref{lem:slowed_progress}.
\end{proof}

%% file: pocketbound.tex

\section{Uniform bounds in a fibered face} \label{sec:uniform_bounds}

In this section we prove the first main theorem of the paper. For a surface $Y$ recall that $|\chi'(Y)| = \max\{|\chi(Y)|, 1\}$.

\begin{theorem}[Bounding projections for $M$] \label{th:bounding_projections_closed}
Let $M$ be a hyperbolic fibered $3$-manifold with fibered face $\FF$. Then for any fiber $S$ contained in $\R_+ \FF$ and any subsurface $Y$ of $S$
\[
 |\chi'(Y)| \cdot (d_Y(\lambda^- ,\lambda^+) -16D) \le 2D \: \vert \FF \vert,
\]
where $|\FF|$ is the number of tetrahedra of the veering triangulation associated to $\FF$.
\end{theorem}


The proof of \Cref{th:bounding_projections_closed} requires the construction of an embedded subcomplex of $(\cM,\tau)$ corresponding to $Y$ whose size is roughly $ |\chi'(Y)| \cdot d_Y(\lambda^- ,\lambda^+)$.  

\subsection{Pockets and the approach from the fully-punctured case}
Let us assume from now on that $d_Y(\lambda^-,\lambda^+)$ is sufficiently large that, by
\Cref{thm: tau-compatible}, $Y$ is $\tau$-compatible. We first describe an approach directly extending the
argument from \cite{veering1}, and explain where it runs into trouble.  

Recall from \Cref{sec:sections} the definition of $T(\partial_\tau Y)$, and its top and bottom sections $T^+$ and $T^-$.  For any two sections $T_1, T_2\in T(\partial_\tau Y)$ we have the region $U(T_1,T_2)$ in $\N$ between them. The closure of the open subsurface $\int_\tau(Y)$ in $X$ is the subspace $\mathcal{R}_Y$ which is the image of $Y_\tau \subset X_Y$ under the covering map $X_Y \to X$. 
Using this
we further define
\begin{equation}\label{pocket Y}
U_Y(T_1,T_2) = U(T_1,T_2) \intersect \Pi^{-1}(\mathcal{R}_Y)
\end{equation}
to be the region between $T_1$ and $T_2$ that lies above $Y$. Note that $U_Y(T_1,T_2)$ is a subcomplex of $\N$.
We sometimes call this the \define{pinched pocket for $Y$} (between
$T_1$ and $T_2$). We let 
\begin{equation}\label{maximal pocket}
  U_Y = U_Y(T^-,T^+)
\end{equation}
denote the \define{maximal pocket for $Y$}.

By \Cref{prop:closed_distance}, $d_Y(T^-,T^+)$ is close to $d_Y(\lambda^-,\lambda^+)$. Thus
when these are sufficiently large, by \Cref{lem:slowed_progress} we obtain a lower bound
on the number of transitions between 
$T^-$ and $T^+$ in $\N$ that project to $\mathcal{R}_Y$, and in particular a lower bound on the number of
tetrahedra in $U_Y$. If $U_Y$ were to embed in $\cM$  (equivalently if $U_Y$ were disjoint
from all its translates by $\Phi$), this would give us what we need.
It is not in general true, so instead we must restrict to a suitable sub-region of $U_Y$. 

In the fully-punctured case, we can use the fact that, since every edge of
$\tau$ represents an element of $\A(S)$, any intersection between $U_Y$ and $\Phi^k(U_Y)$
would project to a non-empty $\pi_Y(\phi^k(\partial Y))$. \Cref{phi n of boundary Y}
implies that,  for $k>0$, whenever $\pi_Y(\phi^k(\partial Y))$ is non-empty it is close to $\pi_Y(\lambda^-)$. Now 
by restricting $U_Y$ to a subcomplex $V_Y$ whose top and bottom surfaces are sufficiently
far in $\A(Y)$ from $\lambda^+$ and $\lambda^-$, and applying this argument to $V_Y$, we
see that $V_Y$ cannot meet $\Phi^k(V_Y)$ at all.

In the general situation, since some singularities of $q$ are not punctures, not every collection
of $\tau$-edges is essential and we are faced with the possibility that $U_Y$ and
$\Phi^k(U_Y)$ can intersect in large but homotopically inessential subcomplexes whose
location is hard to control. This is the main difficulty.

\subsection{Isolation via $\phi$-sections}
We begin, therefore, with the following construction. Let $T_0$ denote a $\phi$-section (\Cref{sec:sections}),
so that $\phi^k(T_0) \le T_0$ for $k>0$. 
Define
\begin{equation}\label{NY def}
N = N_Y = \min \{ i > 0 : \phi^i(Y) \text{ overlaps } Y\}.
\end{equation}
Now consider the (possibly empty) region
\begin{equation}\label{R def}
R(T_0) = \int(U(\Phi^N(T_0), T_0)) \intersect \int(U_Y).
\end{equation}
It easily satisfies the embedding property:

\begin{proposition} \label{prop:isolated_embed}
The restriction of the covering map $\N \to \ring M$ to 
$R(T_0)$ is an embedding.
\end{proposition}

\begin{proof}
We must show that  $\Phi^i (R(T_0))$ is disjoint from
$R(T_0)$ for all $i > 0$ (the case for $i<0$ immediately
follows). For $i \ge N$  this is taken care of 
by the first term of the intersection since $T_0$ was chosen to be a $\phi$--section.

For $0<i<N$, we have that $\int_\tau(Y)$ and $\int_\tau(\phi^i(Y))$
are disjoint since the surfaces have no essential intersection, using
\Cref{lem:overlap_tau}. Therefore 
the interiors of $U_Y$ and
$\Phi^i(U_Y)$ are disjoint as well.
\end{proof}

What remains now is to choose $T_0$ so that a lower bound on $d_Y(\lambda^-,\lambda^+)$
implies a lower bound on the number of tetrahedra in $R(T_0)$. That is, we want to view
$R(T_0)$ as the interior of a ``pocket'' between two sections, whose projections to
$\A(Y)$ are close to $\lambda^\pm$.  For this we will need to
describe $R(T_0)$ from a different point of view. 

\subsection{$Y$-projections of sections}
For \emph{any} section $T$ of $\N$, there is a corresponding section
$T^Y \in T(\partial_\tau Y)$ obtained by pushing $T$ below $T^+$ and
above $T^-$. More formally,
$$
T^Y = T^+ \vmin (T^- \vmax T) = T^- \vmax (T^+ \vmin T).
$$
To see what this does it is helpful to consider it along the fibers of $\Pi$ which we
recall are oriented lines. For any $[a,b]\subset \R$, the map $x \mapsto b \vmin ( a\vmax
x)$ is simply retraction of $\R$ to $[a,b]$, and is equal to $x \mapsto a\vmax (b\vmin
x)$.

Now we can use this projection to extend the notation
$U_Y(T_1,T_2)$ (defined in (\ref{pocket
  Y})) to sections which are not necessarily in $T(\partial_\tau Y)$ by setting
\begin{equation}\label{pocket Y general}
\widehat U_Y(T_1,T_2) = U_Y(T_1^Y,T_2^Y) \subset U_Y.
\end{equation}

This construction is related to the region $R(T_0)$ defined in (\ref{R def}) by the
following lemma:

\begin{lemma} \label{lem:int_int}
  \begin{equation}\label{eq int int}
    \int(\widehat U_Y(T_1,T_2)) = \int (U(T_1, T_2)) \cap \int(U_Y).
  \end{equation}
\end{lemma}

\begin{proof}
First consider what happens in each $\Pi$-fiber, which is just a statement about
projections in $\R$: If $J = [s,t]$ is an interval in $\R$ we have, as above, the
retraction  $\pi_J(u) = s\vmax (t\vmin u)$, and for any other interval $I$ we immediately
find
\begin{equation}\label{projection intersection in R}
  \int(\pi_J(I)) = \int(I)\intersect \int(J).
\end{equation}
To apply this to our situation note first that both the left and right hand sides of the
equality (\ref{eq int int}) are contained in $\Pi^{-1}(\int_\tau (Y))$. This is because $U_Y$ is in
$\Pi^{-1}(\mathcal{R}_Y)$ and in 
$U(T^-,T^+)$,
which means that for each $x\in \mathcal{R}_Y \ssm \int_\tau(Y) =  \partial_\tau Y$, $U_Y\intersect
\Pi^{-1}(x)$ is a single point, and hence not in the interior.

Now since $\Pi$ is a fibration, for any two sections $T,T'$ and region $Z$ of the form
$U(T,T') \intersect  \Pi^{-1}(\Omega)$ where $\Omega$ is open, we have
$$
\int(Z) \intersect \Pi^{-1}(x) = \int(Z \intersect \Pi^{-1}(x)).
$$ 
For  $x\in \int_\tau (Y)$,
applying this to the left hand side of (\ref{eq int int}) we see that
$$
\int(\widehat U_Y(T_1,T_2)) \intersect \Pi^{-1}(x) =
\int(\widehat U_Y(T_1,T_2) \intersect \Pi^{-1}(x)).
$$
On the right hand side, we obtain
$$
\int (U(T_1, T_2))  \intersect \Pi^{-1}(x) =
\int(U(T_1,T_2) \intersect \Pi^{-1}(x)) 
$$
and
$$
\int (U_Y)  \intersect \Pi^{-1}(x) =
\int(U_Y \intersect \Pi^{-1}(x)).
$$
This reduces the equality to a fiberwise equality, where it follows from
(\ref{projection intersection in R}).
\end{proof} 

\subsection{Relation between $T^Y$ and $\pi_Y$}
A crucial point now is to show that the operation $T \to T^Y$ does not alter the projection to $\A(Y)$ by too much. 

\begin{proposition} \label{prop:stays_close}
For all sections $T$ of $\N$, $d_Y(T, T^Y) \le 4 D$.
\end{proposition}
Here $d_Y(T,T^Y)$ is meant in the sense of (\ref{dY complexes}).

\begin{proof}
  Write $T$ as a union of three subcomplexes, $T = K_+ \union K_0
  \union K_-$, where
\begin{align*}
  K_0 &= T\intersect U(T^-,T^+), \\
  K_+ &= T \intersect (T \vmax T^+),   \\
  K_- &= T \intersect (T\vmin T^-).
\end{align*}
See \Cref{K-decomp}. Since $T^+,T^- \in T(\partial_\tau Y)$, $\Pi_*(K_0)$ is a subcomplex of $\tau$-edges in $X$ that do not cross $\partial_\tau Y$.



\realfig{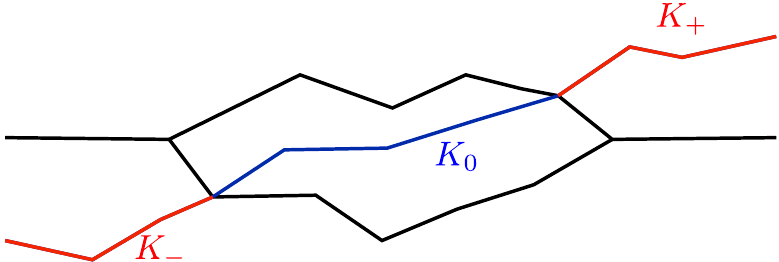}{The section $T$, its intersection with $U(T^-,T^+)$, and the decomposition $T = K_+ \union K_0 \union K_-$.}

\begin{lemma}\label{K0 spine}
  If $\pi_Y^\tau(K_+)$
  and $\pi_Y^\tau(K_-)$ are empty then
  $\Pi_*(K_0)$ contains a punctured spine for $Y$.
\end{lemma}
A punctured spine for $Y$ is a subspace which is a retract
of $Y$ minus a union of disjoint disks. In other words, the conclusion of \Cref{K0 spine} implies that every essential curve in $Y$ is homotopic into $\Pi_*(K_0)$. In particular, $\pi_Y^\tau(K_0)$ is nonempty and so is every $\pi_Z^\tau(K_0)$ for $Z$ a $\tau$-compatible subsurface that overlaps with $Y$.

\begin{proof}
The statement that $\pi_Y^\tau(K_+)$ is empty means that, after
projecting $K_+$ by $\Pi$ into $X$ and intersecting with $\int_\tau(Y)$, we
obtain components which are inessential subcomplexes, meaning they do
not contain any essential curves or proper arcs. (When $Y$ is an annulus, 
this means that no $\tau$-edge from $K_+$ joins opposite sides of the open 
annulus $\int_\tau(Y)$.)

Now $K_+\ssm K_0$ and $K_-\ssm K_0$ are open in $T$ and disjoint, so
their projections to $X$ intersect $\int_\tau(Y)$ in a collection of
disjoint open sets each of which is inessential in the above sense. It
follows that the complement of these open sets, which is
$\Pi(K_0)\intersect \int_\tau(Y)$, intersects every essential curve
and proper arc in $\int_\tau(Y)$. Hence it contains a punctured spine.
\end{proof}




If $T \cap U_Y$ projects to an essential subcomplex of $Y_\tau$, then
the proposition follows since $T$ and $T^Y$ both contain $T\cap U_Y$.

If $T\cap U_Y$ is inessential then, by \Cref{K0 spine}, at least
one of $\pi_Y^\tau(K_\pm)$ is nonempty. Suppose $\pi_Y^\tau(K_+)$ is
nonempty.

Since $K_+ = T \intersect (T \vmax T^+)$, 
we immediately have
\begin{equation*}
  d_Y(T, T \vmax T^+) \le  D.
\end{equation*}
Since $T^+ \le T \vmax T^+$, \Cref{prop:closed_distance}
tells us that
\begin{equation*}
  d_Y(T \vmax T^+,\lambda^+) \le  D.
\end{equation*}
Thus
\begin{equation*}
  d_Y(T,\lambda^+) \le 2 D.
\end{equation*}
Now note that $T^Y\intersect T^+$ 
is equal to the part of $T^+$ lying below $T$, hence its $\Pi$--
projection to $X$ has the same image as $K_+$. It follows that
$\pi_Y^\tau(T^Y\intersect T^+)$ is nonempty, therefore
\begin{equation*}
  d_Y(T^Y,T^+) \le  D. 
\end{equation*}
Since \Cref{prop:closed_distance} again gives us
\begin{equation*}
  d_Y(T^+,\lambda^+) \le  D, 
\end{equation*}
we combine all of these to conclude
\begin{equation*}
  d_Y(T,T^Y) \le 4 D. \qedhere
\end{equation*}
\end{proof}

\begin{remark}\label{rmk:uptop}
Note that the proof of \Cref{prop:stays_close} shows that if $\pi_Y^\tau(K_+)$ is
nonempty, then  $d_Y(T,\lambda^+) \le 2 D$. The corresponding statement also holds if $\pi_Y^\tau(K_-)$ is nonempty.
\end{remark}

\subsection{Finishing the proof}
We assume that $d_Y(\lambda^-,\lambda^+) \ge 10D$.
To complete the argument, we choose a $\phi$-section $T_0$ 
with the property that $3D \le d_Y(T_0,\lambda^+) \le 5D$. Such a section exists by \Cref{lem:AG_sec} and \Cref{prop:closed_distance}. Let $N=N_Y$ be as in (\ref{NY def}).

\begin{lemma} \label{lem:right_place}
With notation as above,
\[
d_Y(\Phi^N (T_0), \lambda^-) \le 2D+11.
\]
\end{lemma}

\begin{proof}
From \Cref{phi n of boundary Y} we have
  \[
d_Y(\phi^N(\boundary Y), \lambda^-) \le 4
\]
From \Cref{prop:proj_close} we have that since $\phi^N(Y)$ overlaps with $Y$,
\[
d_Y(\partial \phi^N(Y),\pi_Y^\tau(\partial_\tau \phi^N(Y))) \le 7.
\]
Since $\phi^N(\boundary_\tau Y) = \boundary_\tau(\phi^N(Y))$, we have
\[
d_Y(\phi^N(\boundary Y),\pi_Y^\tau(\phi^N(\boundary_\tau Y))) \le 7 .
\]
Let $K_0,K_+,K_-$ be subcomplexes of $T_0$ as in the proof of
\Cref{prop:stays_close}. Since $d_Y(T_0,\lambda^+) \ge 3D$ and $d_Y(T_0,\lambda^-) \ge 10D - 5D$
by choice of $T_0$ and the assumption on $d_Y(\lambda^-,\lambda^+)$, 
the proof of \Cref{prop:stays_close} (see \Cref{rmk:uptop}) tells us that both
$\pi_Y^\tau(K^+)$ and $\pi_Y^\tau(K^-)$ are empty.
Hence by \Cref{K0 spine}, $T_0\intersect U_Y$ must contain a punctured
spine for $Y$.

Now since $\phi^N(Y)$ intersects $Y$ essentially, it must be that
$\pi_Y^\tau(\Phi^N(T_0\intersect U_Y))$ is nonempty.
Since the triangulation of $T_0^Y$ contains both $T_0\intersect U_Y$ and
$\boundary_\tau Y$, 
we have
$$
d_Y(\pi_Y^\tau(\Phi^N(T_0\intersect
U_Y)),\pi_Y^\tau(\Phi^N(\boundary_\tau Y))) \le  D.
$$
Combining these and observing that $T_0\intersect U_Y \subset T_0$ we obtain the desired inequality.
\end{proof}


We now define the \define{isolated pocket} for $Y$ to be 
\begin{equation}\label{isolate}
V = V_Y = \widehat U_Y(\Phi^N(T_0), T_0). 
\end{equation}
\Cref{lem:int_int} implies that the interior of $V$ is equal to $R(T_0)$, and
\Cref{prop:isolated_embed} therefore implies the following corollary.

\begin{corollary}\label{cor:embed}
The covering map $\N \to \cM$ embeds
 $\int(V)$ in $\ring M$.
\end{corollary}

Thus we can complete the proof of \Cref{th:bounding_projections_closed} with the following proposition.

\begin{proposition}
The isolated pocket for $Y$ satisfies
\[
|V| \ge \frac{1}{2D}\big |\chi'(Y)| \cdot \big (d_Y(\lambda^-,\lambda^+) - 16D \big).
\]
\end{proposition}
\begin{proof}
By definition,
\[
V = U_Y(\Phi^N(T_0)^Y, T_0^Y) \subset U_Y.
\]
Moreover, we claim that the sections defining this region satisfy the following:
\begin{align} \label{claim:top/bottom}
d_Y(T_0^Y,\lambda^+) \le 9 D \quad \text{and} \quad d_Y(\Phi^N(T_0)^Y,\lambda^-) \le 7 D.
\end{align}

The first inequality follows directly from the assumption that $d_Y(T_0,\lambda^+) \le 5D$ and \Cref{prop:stays_close}. The second inequality follows from \Cref{lem:right_place} and another application of \Cref{prop:stays_close}. (Here we have used that $11 \le D$.)

So we see that $d_Y(\Phi^N(T_0)^Y, T_0^Y) \ge  d_Y(\lambda^-,\lambda^+) - 16D $. 
Hence, to prove the proposition, it suffices to show that 
\begin{align}\label{eq:pocket_size}
\frac{|V|}{|\chi'(Y)|} \ge \frac{d_Y(\Phi^N(T_0)^Y, T_0^Y)}{2D}.
\end{align}

Set $d = d_Y(\Phi^N(T_0)^Y, T_0^Y)$.
If $Y$ is not an annulus, then \Cref{lem:slowed_progress} implies that at least $|\chi'(Y)| \cdot d/2D$ upward tetrahedron moves through tetrahedra of $U_Y$ are needed to connect $\Phi^N(T_0)^Y$ to $T_0^Y$. 
If $Y$ is an annulus, the same is true since $|\chi'(Y)| =1$ and any triangulation of $Y$ has at least two edges joining opposite boundary components.
As each of these tetrahedra lie in $V$ by definition, this establishes \Cref{eq:pocket_size} and completes the proof.
\qedhere
\end{proof}

%% file: dichotomy.tex

\section{The subsurface dichotomy}
\label{sec:dichotomy}

In this section we prove the second of our main theorems:

\begin{theorem}[Subsurface dichotomy] \label{th:closed_sub_dichotomy}
Let $M$ be a hyperbolic fibered $3$-manifold and let $S$ and $F$ be fibers of $M$ which are contained in the same fibered face. If $W \subset F$ is a subsurface of $F$ then either $W$ is 
homotopic through surfaces transverse to the flow
to an \emph{embedded} subsurface $W' \subset S$ of $S$ with 
$$d_{W'} (\lambda^-,\lambda^+) = d_W(\lambda^-,\lambda^+) $$
or the fiber $S$ satisfies
$$9D \cdot |\chi(S)| \ge  d_W(\lambda^-,\lambda^+) -16D.$$
\end{theorem}

\subsection*{Punctures and blowups.} 
Fix a fibered face $\FF$ of $M$ and denote the corresponding 
veering triangulation of $\cM$ by $\tau$.
Starting with a fiber $F$ of $M$ in the face $\FF$
let $\mathring F$ be the
fully punctured fiber, that is $\mathring F = F \ssm \sing(q)$. Also 
let $\N_F$ be the infinite cyclic cover of $\cM$ corresponding to the 
fiber $\cF$ together with its veering triangulation (the preimage of $\tau$), as in \Cref{sec:sections}.


For any section $T$ of $\N_F$, let 
 $h_{\mathring F, T} \colon \mathring F \to \mathring M$ be the
simplicial map obtained by composing the section with the covering
map $\N_F \to \cM$.
We want to describe a natural way
to obtain a map $h_{F, T} \colon F\to M$ by filling in punctures:

Let $\check{F}$ be the partial compactification of $\mathring F$ to a surface
with boundary obtained by adjoining the links of ideal vertices in
$\sing(q) \ssm\PP$, in the
simplicial structure on $\mathring F$ induced from $T$. 
In other words we add a circle to each puncture. Similarly
let $\check M$ be the manifold with torus boundaries obtained by
adding links for the ideal vertices of $\tau$ associated to singular
orbits of $\sing(q)\ssm \PP$. Then $h_{\mathring F, T}$ extends
continuously to a proper map $h_{\check F, T} \colon \check F \to \check M$.

We obtain (a copy of) $M$ from $\check M$ by adjoining solid tori to
the boundary components, and a copy of $F$ from $\check F$ by adjoining disks.
Now by construction (since $\mathring F$ comes from puncturing the fiber
$F$ of $M$) the boundary components of $\check F$ map to meridians of the
tori. It follows that $h_{\check F, T}$ can be extended to a map
$h_{F, T} \colon F\to M$ which maps the disks into the solid tori.

Now assume that $W$ is a $\tau$-compatible subsurface of $F$. When $T \in T(\partial_\tau W)$, we can restrict this construction to $W$ as follows: 
First, let $\mathring W_\tau$ be subsurface obtained from $W_\tau \subset F_W$ by puncturing along all singularities. Then the covering $F_W \to F$ restricts to a map $\mathring W_\tau \to \mathring F$ which sends
 the interior of $\mathring W_\tau$ homeomorphically 
 onto  $\int_\tau(W) \ssm \sing(q)$. Composing this map with $h_{\mathring F, T}$, we obtain $h_{\mathring W, T}$. (Note that since $T \in T(\partial_\tau W)$, $T$ naturally induces an ideal triangulation of $\mathring W_\tau$.)
Restricting the above construction to $\mathring W_\tau$, we obtain $\check W$, $h_{\check W, T}$ and $h_{W, T}$. This is done in such a way that each ideal point of $\partial W_\tau$ is replaced with an arc when forming $\check W$ and a ``half-disk'' when forming $W$.


If $S$ is another fiber in the same face, together with a given section of $\N_S$, we define $h_{\mathring S}$,
$h_{\check S}$ and $h_{S}$ in the same way. (Since we will not vary the section of $\N_S$, we do not include it in the notation.)

\subsection*{Intersection locus}
Now consider the locus $h_{\mathring W, T}^{-1}(h_{\mathring
  S}(\mathring S))$, which is a $\tau$-simplicial subcomplex of
$\mathring W_\tau$.
Its completion in $W_\tau \subset F_W$ (with respect to the underlying $q$-metric on $F$) is obtained by adjoining points of $\sing(q) \ssm
\PP$, and we say that this completion is {\em inessential} 
if each of its components can be deformed to a point or to a boundary (or
puncture) of $W$.
We say it is essential if it is not inessential. 
In other words, the completion is
essential if its $1$-skeleton is an essential proper graph of $W_\tau$.

\begin{lemma}[Essential intersection] \label{lem:essential intersection}
Let $F$ and $S$ be fibers in the same face $\FF$ of $M$.
Suppose $W \subset F$ is $\tau$-compatible subsurface of $F$ 
such that $\pi_1(W)$ is not contained in $\pi_1(S)$.
Then, for any section $T$ of $\N_F$ in $T(\partial_\tau W)$,
the completion of
$h_{\mathring W, T}^{-1}(h_{\mathring
  S}(\mathring S))$ in $W_\tau$
is essential.
\end{lemma}

\begin{proof}
After obtaining the blowups and maps, as above, we first claim that 
$$
h_{W, T}^{-1}(h_S( S))
$$ 
is essential. The argument for this is similar to \cite[Lemma 2.9]{veering1}, although 
there it was assumed that $h_W = h_{W,T}$ embeds $W$ into $M$. In details,
if the preimage was not essential
then each component is
homotopic into a disk, or homotopic into the ends of $W$. It follows that
$h_W$ is homotopic to a map $h'_W$ whose image misses $h_S(S)$
entirely -- just precompose with an isotopy of $W$ into itself which
lands in the complement of  $h_W^{-1}(h_S( S))$.

But if $h'_W$ misses $h_S(S)$ we can conclude that $\pi_1(W)$ is in
$\pi_1(S)$: letting $\eta$ denote the cohomology class dual to $S$ in $M$,
that fact $h'_W$ misses $h_S(S)$ implies that $\eta$ vanishes on $\pi_1(W)$. 
Hence, if $h_W^{-1}(h_S( S))$ is inessential, then $\pi_1(W) \le \pi_1(S)$.

Now note that 
$h_W^{-1}(h_S(S))$ is contained a small neighborhood of
the completion of $h_{\mathring W}^{-1}(h_{\mathring S}(\mathring S))$.
Indeed, each component of the completion of $h_{\mathring W}^{-1}(h_{\mathring S}(\mathring S))$ can be obtained from a component of $h_W^{-1}(h_S(S))$ by collapsing the adjoined disks back to singularities.
We conclude that the completion of $h_{\mathring W}^{-1}(h_{\mathring S}(\mathring S))$ contains an essential component as well. This is
what we wanted to prove. 
\end{proof}

\subsection*{Sections of $V_W$}
Assume that $d_W(\lambda^-,\lambda^+) \ge 10D$. Then $W$ has an
isolated pocket 
$V_W = U_W(\phi^N(T)^W, T^W) \subset \N_F$, 
 as in \cref{isolate}. By \Cref{cor:embed}, 
the restriction $\int(V_W) \to \cM$ 
of the covering $\N_F \to \cM$
is an embedding.
Now fix a sequence 
\[
\phi^N(T)^W = T_0, T_1, \ldots, T_n, T_{n+1} = T^W
\]
of sections of $T(\partial_\tau W)$ such that $T_i \to T_{i+1}$ is a tetrahedron move
in $U_W$ for $i<n$, and $\Pi_*(T_n)$ and $\Pi_*(T_{n+1})$ restrict to the same triangulation
of $\mathring W_\tau$ (note that $T_n\to T_{n+1}$ is not a tetrahedron move, because the
triangulations can be different outside $W_\tau$). 
That such a sequence exists follows from \Cref{lem:connectivity}.

Using these sections, we define the subcomplex $W_i = T_i \cap U_W$ of $T_i$.
By construction, the tetrahedron between $T_i$ and $T_{i+1}$
lies between $W_i$ and $W_{i+1}$ and
so is contained in $\int (V_W)$.

Denote the map $h_{\mathring W, T_i}$ associated to the 
section $T_i$ by $h_i \colon \mathring W_\tau \to M$.

\vspace{5mm}

With this setup, we can complete the proof of \Cref{th:closed_sub_dichotomy}.

\begin{proof}[Proof of \Cref{th:closed_sub_dichotomy}]
We may assume that $d_W(\lambda^+,\lambda^+) \ge 10D$.

First, suppose that $\pi_1(W)$ is not contained in $\pi_1(S)$.
Then by \Cref{lem:essential intersection} the subcomplex $
h_{i}^{-1}(h_{\mathring  S}(\mathring S))$ of $W_i$ has
essential completion in $W_\tau$
for each $0\le i \le n$.
Denote the $1$-skeleton of the completion of $
h_{i}^{-1}(h_{\mathring  S}(\mathring S))$ by $p_i$
and note that its $\pi_W^\tau$--projection is a 
nontrivial subset of $\A(W)$ 
whose diameter
is bounded by $D$. 
Each transition from  $p_i$ to $p_{i+1}$
corresponds to a tetrahedron move, where $p_i$ contains the bottom
edge and $p_{i+1}$ the top edge. Because the tetrahedra are in $V_W$,
which embeds in $\mathring M$, their top edges in $\cM$ are all distinct (any
edge is the top edge of a unique tetrahedron). Since all these edges
are in the image of $h_{\mathring{S}}(\mathring{S})$
we find that their number is
bounded above by $9|\chi(S)|$. Hence, $n \le 9|\chi(S)|$.


Since $p_0$ and $p_n$ are in the triangulations associated to  the
bottom and top of $V_W$, respectively, we have (\Cref{claim:top/bottom})
\[
d_W(p_0,p_n) \ge d_W(\lambda^-,\lambda^+) - 16D.
\]
For each transition $p_i \to p_{i+1}$ observe that 
$\diam_W(p_i\union p_{i+1}) \le D$ since the proper graph $p_i \cup p_{i+1}$ of  $\int_\tau(W)$ has at most $2|\chi(W)|+1$ vertices (\Cref{cor:diam_bound}).
Combining these facts we obtain
\[
9|\chi(S)| \ge (d_W(\lambda^-,\lambda^+)-  16D)/D
\]
which completes the proof in this case.

Otherwise $\pi_1(W) \le \pi_1(S)$, and we finish the proof as in \cite[Theorem 1.2]{veering1} using \Cref{thm:immersion bound lambda} in place of the special case obtained there. 
First, recall that by \cite[Lemma 2.8]{veering1}, the quantity $d_W(\lambda^-,\lambda^+)$
depends only on the conjugacy class $\pi_1(W) \le \pi_1(M)$ and the fibered face $\FF$. If
we lift $W$ to the $S$--cover of $M$, we see that projecting along flow lines gives an
immersion $W \to S$ that induces the inclusion $\pi_1(W) \le \pi_1(S)$ up to
conjugation. Since  $d_W(\lambda^-,\lambda^+) \ge 10D \ge 37$,  \Cref{thm:immersion bound
  lambda} implies that the map $W \to S$ factors up to homotopy through a finite cover $W \to W'$ for a subsurface $W'$ of $S$.
Since $\pi_1(W) < \pi_1(F)$, the map $\pi_1(W') \to \pi_1(M)/\pi_1(F)$ factors through
$\pi_1(W')/\pi_1(W)$ which is finite. Since $\pi_1(M)/\pi_1(F)=\Z$, this implies $\pi_1(W')$
maps to the identity in $\pi_1(M)/\pi_1(F)$ so $\pi_1(W')<\pi_1(F)$. 
But this implies that $\pi_1(W) = \pi_1(W')$. Hence, the cover $W \to W'$ has degree $1$ and so
$d_W(\lambda^-,\lambda^+) = d_{W'}(\lambda^-,\lambda^+)$. This completes the proof.
\end{proof}